\documentclass{amsart}
\usepackage{verbatim}
\usepackage{epsfig}
\usepackage{graphicx}
\usepackage{mathrsfs}
\usepackage{amsmath, amssymb, amsfonts, euscript, enumerate}
\usepackage{appendix}

\usepackage{hyperref}
\usepackage{txfonts}
\usepackage{color}

\newcommand{\R}{{\mathbb R}}
\newcommand{\N}{{\mathbb N}}

\newcommand{\supp}{\text{\rm supp}}

\newcommand{\ap}{\alpha}             
\newcommand{\bt}{\beta}

\newcommand{\vep}{\varepsilon}

\newcommand{\ld}{\lambda}            \newcommand{\Ld}{\Lambda}
\newcommand{\sm}{\sigma}             
\newcommand{\vp}{\varphi}

\newcommand{\f}{\frac}

\newcommand{\fL}{{\mathfrak L}}

\newcommand{\fm}{{\mathfrak m}}

\newcommand{\cI}{{\mathcal I}}

\newcommand{\cK}{{\mathcal K}}
\newcommand{\cL}{{\mathcal L}}
\newcommand{\cM}{{\mathcal M}}

\newcommand{\cR}{{\mathcal R}}
\newcommand{\cS}{{\mathcal S}}

         \newcommand{\ep}{\epsilon}
            
\newcommand{\pa}{\partial}

\newcommand{\D}{\nabla}

\newtheorem{thm}{Theorem}[section]
\newtheorem{lemma}[thm]{Lemma}
\newtheorem{cor}[thm]{Corollary}
\newtheorem{remark}[thm]{Remark}
\newtheorem{prop}[thm]{Proposition}
\newtheorem{definition}[thm]{Definition}
\newtheorem{example}[thm]{Example}
\newtheorem{hypo}[thm]{Property}

\theoremstyle{definition}

\title[Integro-differential
operators with regularly varying kernels]
{Regularity   for fully nonlinear integro-differential
operators   with regularly varying kernels
     }

\author[Soojung Kim]{Soojung Kim}
\address{Soojung Kim :
National Institue for Mathematical Sciences, 
 Daejeon, 306-390,  Republic of Korea  }
\email{soojung26@gmail.com; soojung26@nims.re.kr}

\author[Yong-Cheol Kim]{ Yong-Cheol Kim}
\address{Yong-Cheol Kim: Department of Mathematics Education, Korea University, Seoul 136-701,
Republic of Korea }
\email{ychkim@korea.ac.kr}

\author[Ki-Ahm Lee]{Ki-Ahm Lee}
\address{Ki-Ahm Lee:
 School of Mathematical Sciences, Seoul National University,   Seoul 151-747, Republic of Korea \& Center for Mathematical Challenges, Korea Institute for Advanced Study, Seoul,130-722,  Republic of Korea}
\email{kiahm@math.snu.ac.kr}

\begin{document}
\begin{abstract} 
In this paper,  the regularity results  for the   integro-differential operators of the  fractional Laplacian type    by Caffarelli and Silvestre \cite{CS1}    are extended to  those for the  integro-differential
operators associated with symmetric,  regularly varying kernels at zero.  In particular, we obtain  the uniform Harnack inequality and H\"older estimate     of viscosity solutions to  the nonlinear  integro-differential equations  associated  with  the   kernels   $K_{\sigma, \beta}$    satisfying 
 $$
K_{\sigma,\beta}(y)\asymp \frac{ 2-\sigma}{|y|^{n+\sigma}}\left( \log\frac{2}{|y|^2}\right)^{\beta(2-\sigma)}\quad \mbox{near zero}
$$
with respect to $\sigma\in(0,2)$ close to $2$ (for a given      $\beta\in\mathbb R$), where the regularity estimates    do not blow up
 as    the order  $  \sigma\in(0,2)$  tends to   $2.$ 

\end{abstract}

\maketitle

\tableofcontents
	
\section{Introduction}

\subsection{Introduction}

In this paper, we  are concerned with  
 fully nonlinear elliptic    integro-differential operators   associated with   symmetric,  regularly varying kernels at zero.  From   the L\'evy-Khinchine formula, the purely jump processes which allow particles to interact at large scales  
are generated by  the   integral operators in the   form of 
  \begin{equation}\label{eq-main0}
  \cL u(x)=P.V. \int_{\R^n}\left\{u(x+y)-u(x) -\left(\nabla u(x)\cdot y \right)\chi_{B_1(0)}(y)\right\}d\fm(y),
  \end{equation}
  where    a  so-called   L\'evy measure   $\fm$  satisfies 
\begin{equation*}
 \int_{\R^n} \min(1,|y|^2) d\fm(y)<+\infty. 
\end{equation*}  
 Since the operators are given in too much generality, we    therefore restrict ourselves to considering only  the   operators given by   symmetric    kernels.   
In this case,   the operator  \eqref{eq-main0} can be written as   
  \begin{equation}\label{eq-main-intro}
\cL u(x)=\int_{\R^n}\left\{u(x+y)+u(x-y)-2u(x)\right\}K(y)\,dy,
\end{equation} 
where a symmetric  L\'evy measure   $\fm$  in \eqref{eq-main0}   is  given by a symmetric    kernel $K(y)=K(-y). $   
We note that the value of $\cL u(x)$ is well-defined when $u$ is bounded in $\R^n$ and $C^{1,1}$ at $x$ (see Definition \ref{def-C^{1,1}}). Nonlinear integro-differential operators  associated with the linear integro-differential operators above  arise naturally in the study    of  
   the  stochastic
control theory related to
\begin{equation*}
\cI u(x)=\sup_{\ap}\cL_{\ap}u(x),
\end{equation*}
and  game theory associated with
\begin{equation*}
\cI u(x)=\inf_{\bt}\sup_{\ap}\cL_{\ap\bt}u(x). 
\end{equation*} 
To study  uniform  regularity      for   such nonlinear integro-differential  operators,        
 the concept of ellipticity  for integro-differential  operators with respect to a class  $\fL$ of the  linear, integro-differential operators    \eqref{eq-main-intro}  
 was   introduced by Caffarelli and Silvestre \cite{CS1}; see \cite{CC}  for   elliptic second-order differential  operators.  
 In fact,  the concept of ellipticity for   integro-differential operators $\cI$     is characterized by the following property:
$$\inf_{\cL  \in\fL}\cL v(x) \leq \cI[u+v](x)-\cI u(x)\leq \sup_{\cL\in\fL}\cL v(x).  $$ 
On the basis of this idea, the regularity theory for  fully nonlinear  elliptic integro-differential operators  has been developed    by using analytic techniques  along the lines of the  Krylov and Safonov  \cite{KS,CC}  
  which dealt with  elliptic second-order differential  operators.  We refer   to    \cite{CS1, CS2}   and   references therein for uniform  regularity results for symmetric integro-differential  operators of the fractional Laplacian type, where the regularity estimates do not blow up   as  the order $\sm\in(0,2)$ of the operators   tends  to $2$.  
  In the case when the kernels are   nonsymmetric,  the uniform regularity  results  can be found in      
   \cite{KL1,KL2, LD1}.  We refer to \cite{KL3,LD2}  for results on the regularity of the parabolic  integro-differetial operators.

In this paper, we establish     the  uniform regularity  of viscosity solutions to      
 fully nonlinear  elliptic  integro-differential equations associated with symmetric,  regularly varying kernels at zero. 
We are mainly interested in     the 
        kernels $K$ for the integro-differential operator   \eqref{eq-main-intro}    satisfying    
\begin{equation}\label{eq-kernel1}
 \int_{\R^n} \min(1,|y|^2) K(y)dy<+\infty
\end{equation} 
and 
\begin{equation}\label{eq-kernel2} 
\qquad\qquad\quad (2-\sm)\ld\f{ l (|y|)}{|y|^{n}}\le
K(y)
\le (2-\sm)\Ld\f{ l(|y|)}{|y|^{n}},\quad 0<\ld\leq\Ld<+\infty, 
\end{equation}
 where    $l : (0, +\infty) \rightarrow (0,+\infty) $  is  a   locally bounded,  regularly varying function 
  at zero   with index $ {\color{black}-\sigma \in(-2,0)}$; refer to Appendix \ref{sec-regular-variation} for  regular  variations.      
The   kernels  of the type  \eqref{eq-kernel2} associated with  regularly varying functions at zero  appear  in the study of   the subordinate Brownian motions which are time changed Brownian motions  by     independent  subordinators (i.e.,  nonnegative L\'evy  processes); see e.g.   \cite{KSV} for a potential theory of  subordinate Brownian motions.   
 There are   known results on Harnack inequalities and H\"older estimates for    integro-differential  operators of this type with   probabilistic proofs.  
 In particular,   Kassmann and Mimica \cite{KM} recently  obtained  H\"older type estimates  for the linear  integro-differential operators  with regularly varying kernels at zero with index $-\sm\in[-2,0]$ based on intrinsic scaling properties;  the H\"older type estimates 
 blow up as the order of the operator approaches 2.
   As an extension of    the regularity  results   by Caffarelli and Silvestre  \cite{CS1}, we   obtain     uniform regularity results of  viscosity solutions  for a {\color{black} certain} class of  fully nonlinear elliptic  integro-differential operators associated  with symmetric,  regularly varying kernels at zero, which remain uniform as the order $\sm\in(0,2)$   tends to  $2$.


\subsection{Integro-differential operators}\label{subsec-operators}
As mentioned above, we study  uniform   regularity of viscosity solutions  for  a class of   the fully nonlinear elliptic   integro-differential operators associated with the  kernels of the type: 
$$K(y)\asymp (2-\sm)\frac{l(|y|)}{|y|^n}$$
for a regularly varying function  $l$  at zero       with index $-\sm\in(-2,0)$. For the purpose,   we first summarize   the    properties of the   regular variations  that  play an essential role in our analysis.  In the entire article,   a  measurable function $l:(0,+\infty)\to (0,+\infty)$  which stays locally bounded away from $0$ and $+\infty,$ will be commonly assumed to satisfy the following   properties. 

  \begin{hypo}\label{hypo-kernel-l}
 Let  a measurable function $l:(0,1]\to (0,+\infty)$  be locally bounded away from $0$ and $+\infty$. There exist  positive constants $\sm\in(0,2),$ $a_0\geq 1,$  and $\rho\in(0,1)$ satisfying the following. 
 \begin{enumerate}[(a)]

 \item 
There exists  $\delta\in\left[0, \frac{1}{2}\min( 2-\sm ,\sm)\right)\subset\left[0,1\right)$  
 such that 
$$  \frac{l(s)}{l(r)} 
\leq  {\color{black}a_0} \max\left\{  \left(\frac{s}{r}\right)^{-\sm+\delta}, \left(\frac{s}{r}\right)^{-\sm-\delta}\right\}\qquad\mbox{ for $r,s \in(0,1]$  }.$$

\item  Define
 $$L(r):= {\sm}\int_{r}^1 \frac{l(s)}{s}ds. $$
Then we have   for any $r\in (0,\rho)$
$$\frac{1}{2 } \leq \frac{L(r)}{l(r)}\leq 2.$$


 \item We   assume that $l(1)=1.$
\end{enumerate}  
\end{hypo} 

Influenced by Kassmann and Mimica \cite{KM}, we introduce the monotone function $L$ above defined by using the given     function $l$   in order to study scale invariant regularity estimates for the  integro-differential operators associated with symmetric, regularly varying kernels at zero.

The function $l$ at infinity  will be commonly  assumed to  satisfy the following property. 
  \begin{hypo}\label{hypo-kernel-l-infty}

Let  a measurable function $l:(0,+\infty)\to (0,+\infty)$  be locally bounded away from $0$ and $+\infty$,  and satisfy   Property \ref{hypo-kernel-l}.  
 There exists a positive constant $a_\infty\geq1 $ such that for some $\delta'\in\left[0, \frac{1}{2}\min( 2-\sm ,\sm)\right)\subset\left[0,1\right)$  
$$  \frac{l(s)}{l(r)} 
\leq  {\color{black}a_\infty} \max\left\{  \left(\frac{s}{r}\right)^{-\sm+\delta'}, \left(\frac{s}{r}\right)^{-\sm-\delta'}\right\}\qquad\mbox{ for $r,s \in[1,+\infty)$  }.$$


\end{hypo} 

Typical examples of the functions satisfying Property \ref{hypo-kernel-l} are 
   regularly varying functions  at zero with index $-\sm\in(-2,0),$ and   Property  \ref{hypo-kernel-l-infty}
is satisfied by assuming that the function $l$  varies regularly at infinity with  index $-\sm\in(-2,0)$; 
refer to  Appendix  \ref{sec-regular-variation}  for  the definition of regular    variations and their   properties.    
\begin{example}[Regularly varying functions]
\begin{enumerate}[(a)]
\item Trivial  examples of such  regularly varying functions  at zero and infinity with index $-\sm\in(-2,0)$ are   
  $$ (2-\sm)r^{-\sm}.$$
 The operator \eqref{eq-main-intro} with the choice above of  the regularly varying function    
   turns out to be  the well-known  fractional Laplacian operator $-\left(-\Delta\right)^{\sm/2}$ defined as 
$$-\left(-\Delta\right)^{\sm/2}u(x):= (2-\sm )\int_{\R^n} \frac{u(x+y)+u(x-y)-2u(x)}{|y|^{n+\sm}}dy,$$
 which converges to the Laplacian operator   as the order  $\sm\in(0,2)$ approaches $ 2.$
    We note that the factor $(2-\sm)$ enables us to obtain  second-order differential operators  as the  limits of integro-differential operators (see  \cite{DPV, CS1}, for example)   and hence uniform regularity results as the order $\sm\in(0,2)$ goes to the classical one.  
      \item   Among nontrivial examples of regularly  varying functions $l$ at zero with index $-\sm$  are  functions which    are equal to the  following functions near zero (see \cite{BGT}):  
      $$ r^{-\sm} \left(\log\frac{2}{r}\right)^{\beta},  r^{-\sm} \left(\log\frac{2}{r^2}\right)^{\beta},\quad\mbox{and}\quad r^{-\sm} \left(\log\log\frac{2}{r}\right)^{\beta} \qquad \mbox{  for   $\beta\in\R$}.$$ 
 \item    The following functions are non-logarithmic   regularly varying functions $l$ at zero with index $-\sm$:
$$r^{-\sm} \exp\left( \left(\log \frac{2}{r}\right)^{\beta} \right)\qquad\mbox{for  $\beta\in(0,1)$,}$$
and $$r^{-\sm} \exp\left(  \log \frac{2}{r} \left/ \log\log\frac{2}{r}\right. \right).$$

\end{enumerate}\end{example}

 For a  certain class of regularly varying functions,  the constants $a_0\geq 1,$   $\rho\in(0,1)$ and  $a_\infty$ in Properties  
\ref{hypo-kernel-l} and \ref{hypo-kernel-l-infty} can be selected  uniformly.  
  Let      $\sm_0\in(0,2),$ and let a locally bounded function   $l_0:(0,+\infty)\to(0,+\infty)$   be a slowly      varying function     at zero and infinity  which varies regularly at zero and infinity with index $0$ from the definition;  see Appendix \ref{sec-regular-variation}.  
 For $\sm\in[\sm_0,2), $ define a regularly varying function $l_\sm$       at zero and infinity  with index $-\sm\in(-2,-\sm_0]$ by 
 \begin{equation}\label{eq-l-sigma-intro}
 l_\sm(r):=r^{-\sm}l_0(r)^{2-\sm},\quad\forall r>0.
 \end{equation}
Making use of  a theory of regular variations,   we shall prove in Proposition \ref{prop-l-sigma} that 
the function $l_\sm$   satisfies Properties \ref{hypo-kernel-l} and \ref{hypo-kernel-l-infty}  with uniform constants $a_0, a_\infty\geq 1$ and $\rho\in(0,1)$ for   $\sm\in[\sm_0,2),$  where the uniform constants depend   only on dimension $n$,   $\sm_0\in(0,2)$, and a  given  slowly varying  function $l_0$ at  zero and infinity.  Regarding the regularly varying functions of the type \eqref{eq-l-sigma-intro}, we remark      that  for $\sm\in(0,2)$ and $\beta\in(-\sm,2-\sm),$ the kernel 
  \begin{align*}
  \displaystyle K(y)&\asymp (2-\sm) |y|^{-n-\sm}{ \left(\log\frac{1}{|y|^2}\right)^{\beta/2}}  \quad\mbox{near zero}
  \end{align*} 
   associated with the regularly varying function  $ l(r)=r^{-\sm} \left(\log\frac{1}{r^2}\right)^{\beta/2} $   describes      the        asymptotic behavior        of     the jumping kernel  at zero  of the subordinate process  which has the  characteristic    exponent    $\phi(s):= s^{\sm}(1+\log s^2  )^{\beta/2}$; refer to  a potential theory of subordinate Brownian motions \cite{KSV}.

To investigate  a   class of fully  nonlinear elliptic integro-differential  operators associated with symmetric, regularly varying kernels at zero, 
 let  $0<\ld\leq\Ld <+\infty,$ and  let   a function $l:(0,+\infty)\to (0,+\infty)$ satisfy  Properties \ref{hypo-kernel-l} and \ref{hypo-kernel-l-infty}.    
Owing to Properties \ref{hypo-kernel-l} and \ref{hypo-kernel-l-infty},    the following   properties of  the given function $l$ concerning the symmetric integro-differential  operators are   obtained; the  proof   can be  found  in Section \ref{sec-harnack}. 
 \begin{lemma} \label{lem-kernel-l-intro}
  Let  a measurable function $l:(0,+\infty)\to (0,+\infty)$  be locally bounded away from $0$ and $+\infty$, and  
satisfy  Properties \ref{hypo-kernel-l} and \ref{hypo-kernel-l-infty}  with   positive constants $\sm\in(0,2), a_0\geq 1, a_\infty\geq 1,$  and $\rho\in(0,1).$ Then we have the following:
 \begin{enumerate}[(a)]
 \item 
$$   {\color{black}\frac{1}{2a_0}}\frac{ r^2l(r)}{2-\sm} \leq  \int_0^rsl(s)ds \leq {\color{black}2a_0}\frac{ r^2l(r)}{2-\sm}, \quad\forall r\in(0,1],$$



\item
      \begin{equation*}
       {\sm}\int_1^{\infty}\frac{l(s)}{s}ds\leq {\color{black}2a_\infty}.
      \end{equation*}
\end{enumerate}
\end{lemma} 
Now, let  
$\fL\left(\ld,\Ld, l\right)$ 
 denote the class  of the following linear integro-differential  operators 
  with the  kernels $K$:
\begin{equation*}
\cL u(x)=\int_{\R^n}\mu(u,x,y)  K(y) dy,
\end{equation*}
where $\mu(u,x,y):=u(x+y)+u(x-y)-2u(x)$
and 
\begin{equation}\label{eq-kernel2-intro-op} 
   (2-\sm)\ld\f{ l (|y|)}{|y|^{n}}\le
K(y)
\le(2-\sm)\Ld\f{ l(|y|)}{|y|^{n}}.
\end{equation}
 One can check   that the   kernels $K$ satisfying     \eqref{eq-kernel2-intro-op}   with   Properties \ref{hypo-kernel-l}  and \ref{hypo-kernel-l-infty}  satisfy \eqref{eq-kernel1} by using  Lemma \ref{lem-kernel-l-intro}.   
As mentioned in  the introduction, 
 we are     concerned with    the  nonlinear  integro-differential operator $\cI$ in the form of 
 \begin{equation}\label{eq-Bellman}
  \cI u:= \inf_\beta \sup_{\alpha}\cL_{\alpha \beta} u  
 \end{equation} 
 for some  $\cL_{\alpha \beta}\in \fL(\ld,\Ld,l).$ 
As extremal cases of  such nonlinear  integro-differential operators,  
       the Pucci type   extremal  operators with respect to the 
class $\fL(\ld,\Ld, l)$ 
  are  defined as 
\begin{equation}\label{eq-Pucci}
\begin{split}
\cM^+_{  \fL\left(\ld,\Ld, l\right)} u&:=\sup_{\cL\in\fL\left(\ld,\Ld, l\right)}\cL u,\\
\cM^-_{  \fL\left(\ld,\Ld, l\right)}u&:=\inf_{\cL\in\fL\left(\ld,\Ld, l\right)}\cL u.
\end{split}
\end{equation}
According to  Lemma \ref{lem-bellman-property}, 
 the integro-differential operator $\cI$ of the  inf-sup type  in  \eqref{eq-Bellman}  is elliptic with respect to $\fL(\ld,\Ld,l)$ in the nonlocal sense,  which, in particular,  implies     
$$\cM^-_{ \fL(\ld,\Ld,l)}  u\leq \cI u   \leq \cM^+_{\fL(\ld,\Ld, l)} u,$$ 
where we refer to Definition \ref{def-elliptic-op} for the nonlocal  notion of the ellipticity. 
Thus  we shall  deal with  a large class  of  the   integro-differential operators   defined in terms of  the Pucci type extremal  operators  so as to 
   establish  uniform  regularity estimates for the  fully nonlinear  elliptic integro-differential operators associated with symmetric, regularly varying kernels at zero.    %
 


\subsection{Main results}
Now  we present  our main results which extend the uniform regularity results of Caffarelli and Silvestre \cite{CS1}. 
 Below and hereafter, 
  we denote $B_R:=B_R(0)$ for $R>0.$
 \begin{thm}[Harnack inequality]\label{thm-Harnack-regularly-varying-kernel-intro}
Let $\sm_0\in(0,2)$ and  let a measurable function  $l:(0,+\infty)\to (0,+\infty)$  be locally bounded   away from $0$ and $+\infty,$ and satisfy  Properties \ref{hypo-kernel-l} and \ref{hypo-kernel-l-infty} with the   positive constants $\sm\in[\sm_0,2)$,  $a_0\geq 1, a_\infty\geq 1,$  and $\rho\in(0,1).$ 
For $0<R<1,$ and $C_0>0,$ let $u\in C(B_{2R})$ be a  bounded, nonnegative   function in $\R^n$ such that  
 $$\cM^-_{  \fL\left(\ld,\Ld, l\right)} u\leq C_0 \quad\mbox{and}\quad\cM^+_{  \fL\left(\ld,\Ld, l\right)}u\geq -C_0\quad\mbox{in $B_{2R}$}$$ in the viscosity sense.  
 Then there exist uniform constants $C>0$ and  $\rho_0\in(0,1)$   such that 
$$\sup_{B_{ {R}}}  u\leq C\left( \inf_{B_{ {R}}}u+ \frac{C_0}{L(\rho_0 R)}\right),$$
 where   $$L(r):= {\sm}\int_{r}^1 \frac{l(s)}{s}ds \quad\forall 0<r<1,$$
and $C>0$ and  $\rho_0\in(0,1)$   are  uniform constants depending only on  $n,\ld,\Ld,   \sm_0, $   $ a_0,a_\infty,$ and $ \rho$.  
\end{thm}

\begin{thm}[H\"older estimate]\label{thm-Holder-intro}
 Under the same   assumption as in Theorem \ref{thm-Harnack-regularly-varying-kernel-intro}, 
  let $u\in C(B_{2R})$ be a bounded   function in $\R^n$ such that  
 $$\cM^-_{  \fL\left(\ld,\Ld, l\right)} u\leq C_0 \quad\mbox{and}\quad\cM^+_{  \fL\left(\ld,\Ld, l\right)}u\geq -C_0\quad\mbox{in $B_{2R}$}$$ in the viscosity sense.  
Then 
we have 
$$ R^{\ap}\,[u]_{\ap, B_{R}}\leq C\left( \|u\|_{L^{\infty}(\R^n)}+\frac{C_0}{L(\rho_0 R)}\right), $$
 where      $[u]_{\ap, B_{R}}$ stands for  the $\ap$-H\"older seminorm on $B_R$, and 
   the  uniform constants $\ap\in(0,1),$     $C>0$ and $\rho_0\in(0,1)$ depend only $n,\ld,\Ld,   \sm_0, $   $ a_0,a_\infty,$ and $ \rho$.  
\end{thm}

\begin{remark}\label{rmk-Harnack}
{\rm
(i) According to Theorems \ref{thm-Harnack-regularly-varying-kernel-intro} and  \ref{thm-Holder-intro}, the Harnack inequality and H\"older estimate hold   for viscosity solutions to  the fully nonlinear   elliptic  integro-differential equations  with respect to $\fL(\ld,\Ld,l).$ In fact,  if $u$ is a   viscosity solution to $\cI u=f $ in $B_{2R}$ for  an elliptic integro-differential operator with respect to $\fL(\ld,\Ld,l)$  and $f\in L^{\infty}(B_{2R}),$ then $u$ satisfies
$$ \cM^-_{\fL(\ld,\Ld,l)}u\leq \|f\|_{L^{\infty}(B_{2R})}+|\cI0|,\quad\mbox{and}\quad\cM^+_{\fL(\ld,\Ld,l)}u \geq -\|f\|_{L^{\infty}(B_{2R})}-|\cI0|$$
in the viscosity sense. Thus, applying Theorems \ref{thm-Harnack-regularly-varying-kernel-intro}  and  \ref{thm-Holder-intro}, the regularity results follow. 

(ii)   For  any   regularly varying   function $l $    at zero  and infinity with index $-\sm\in(-2,0)$ which stays away from 0 and $+\infty$,  we   obtain the Harnack inequality and H\"older estimate     for the    elliptic  integro-differential operators  with respect to $\fL(\ld,\Ld,l)$   as a corollary. 
Here,  the constants in the regularity    estimates above   depend only on  $n,\ld,\Ld, $ and the given   regularly varying function $l. $
 }
 \end{remark}
 
 Making use of Theorem \ref{thm-Holder-intro}, 
 we establish  the $C^{1,\alpha}$ estimate  for the fully nonlinear  elliptic integro-differential operators associated  with regularly varying kernels at zero and infinity provided that the kernels  satisfy a cancelation property at infinity; see Subsection \ref{subsec-C1alpha}. Furthermore,  the H\"older estimate
for the elliptic integro-differential operators associated with truncated kernels at infinity is also obtained in Subsection \ref{subsec-truncated}, which is important  for applications.  In fact, 
 the assumption of the kernels at infinity   for the  H\"older estimate   in Theorem \ref{thm-Holder-intro}  can be weakened   replacing     Property \ref{hypo-kernel-l-infty} by the boundedness of the integral  at infinity 
\begin{equation}\label{eq-integral-infty}
(2-\sm)\int_1^{\infty} \frac{l(s)}{s}ds \leq a_\infty  
\end{equation}  
for some $a_\infty >0$;
in the following,  we rephrase  Theorem \ref{thm-Holder-truncated}  by assuming \eqref{eq-integral-infty}.   
 

\begin{thm}\label{thm-Holder-int-infty}
Let $\sm_0\in(0,2)$ and  let   $l:(0,+\infty)\to [0,+\infty)$  be a measurable function which is   locally bounded  away from $0$ and $+\infty$  on $(0,1]$, and satisfy  Property \ref{hypo-kernel-l}   with the   positive constants $\sm\in[\sm_0,2)$,  $a_0\geq 1, $  and $\rho\in(0,1).$ We assume that $$(2-\sm)\int_1^{\infty} \frac{l(s)}{s}ds \leq a_\infty\quad\mbox{for some  $a_\infty>0.$}$$
 For $0<R<1,$ and $C_0>0,$ let   $u\in C(B_{2R})$ be a bounded   function in $\R^n$ such that  
 $$\cM^-_{  \fL\left(\ld,\Ld, l\right)} u\leq C_0 \quad\mbox{and}\quad\cM^+_{  \fL\left(\ld,\Ld, l\right)}u\geq -C_0\quad\mbox{in $B_{2R}$}$$ in the viscosity sense.  
Then 
we have 
$$ R^{\ap}\,[u]_{\ap, B_{R}}\leq C\left( \|u\|_{L^{\infty}(\R^n)}+\frac{C_0}{L(\rho_0 R)}\right), $$
 where     
   the  uniform constants $\ap\in(0,1),$     $C>0$ and $\rho_0\in(0,1)$ depend only $n,\ld,\Ld,   \sm_0, $   $ a_0,$   $ \rho$ and $a_\infty$.  

\end{thm}

With the help of   Proposition \ref{prop-l-sigma},   Theorems \ref{thm-Harnack-regularly-varying-kernel-intro} ,    \ref{thm-Holder-intro}, and \ref{thm-Holder-int-infty}  yield  the uniform Harnack inequality and H\"older estimate  for  the fully nonlinear    elliptic  integro-differential operators with respect to the class  $\fL(\ld,\Ld,l_\sm)$  associated with symmetric, regularly varying kernels   of  the  type   \eqref{eq-l-sigma-intro} for $\sm\in[\sm_0,2)\subset(0,2)$, where the regularity estimates do not blow up  as the order $\sm $ goes to $2.$



\begin{thm}[Uniform estimates for   the operators associated with the kernels of the  type   \eqref{eq-l-sigma-intro}]\label{thm-Harnack-l-sigma-intro}
 Let $\sm_0\in(0,2),$ and  let  a measurable function $l_0:(0,+\infty)\to(0,+\infty)$ be  locally bounded away from $0$ and $+\infty$, and    vary slowly  at zero and infinity such that $l_0(1)=1$. 
 For $\sm\in[\sm_0,2), $ define 
 {\color{black}
 \begin{equation}\label{eq-L-sigma}
 \begin{split}
 l_\sm(r)&:=r^{-\sm}l_0(r)^{2-\sm},\quad \forall r\in(0,+\infty),\\ L_\sm(r)&:= {\sm}\int_{r}^1  {s^{-1-\sm}l_0(s)^{2-\sm}}ds,\quad \forall r\in(0,1].
 \end{split}
 \end{equation}}
 \begin{enumerate}[(a)]
\item For $0<R<1,$ and $C_0>0,$ let $u\in C(B_{2R})$ be a  bounded, nonnegative   function in $\R^n$ such that  
 $$\cM^-_{  \fL\left(\ld,\Ld, l_\sm\right)} u\leq C_0 \quad\mbox{and}\quad\cM^+_{  \fL\left(\ld,\Ld, l_\sm\right)}u\geq -C_0\quad\mbox{in $B_{2R}$}$$ in the viscosity sense.     Then we have 
$$\sup_{B_{ R}}  u\leq C\left( \inf_{B_{ {R}}}u+ \frac{C_0}{L_\sm(\rho_0 R)}\right).$$
\item 
Let $u\in C(B_{2R})$ be a bounded   function in $\R^n$ such that  
 $$\cM^-_{  \fL\left(\ld,\Ld, l_\sm\right)} u\leq C_0 \quad\mbox{and}\quad\cM^+_{  \fL\left(\ld,\Ld, l_\sm\right)}u\geq -C_0\quad\mbox{in $B_{2R}$}$$ in the viscosity sense.     Then we have 
$$R^{\ap}\,[u]_{\ap, B_{R}}\leq C\left( \|u\|_{L^{\infty}(\R^n)}+\frac{C_0}{L_\sm(\rho_0 R)}\right) ,$$
 where  $C>0,$    $\rho_0\in(0,1)$ and $\ap\in(0,1)$  are  uniform constants depending only on  $n,\ld,\Ld,  \sm_0$  and the slowly varying function  $l_0$ at zero and infinity.   
 \end{enumerate}
\end{thm}    
In Theorem \ref{thm-Harnack-l-sigma-intro},  we establish the uniform Harnack inequality  and H\"older estimate  for  a class of  fully nonlinear elliptic  integro-differential operators associated with the  kernels $K_\sm$  in the form of 
\begin{equation*}\label{eq-kernel2-uniform} 
 (2-\sm)\ld\f{ l_0 (|y|)^{2-\sm}}{|y|^{n+\sm}}\le
K_\sm(y)\le (2-\sm) \Ld\f{ l_0(|y|)^{2-\sm}}{|y|^{n+\sm}}. 
\end{equation*} 
In the case when  $l_0\equiv 1,$ we observe that for $\sm\in[\sm_0,2),$
  \begin{align*}
 l_\sm(r)&=r^{-\sm}\\
  L_\sm(r)&=   r^{-\sm}-1  \geq (1-2^{-\sm_0})r^{-\sm}, \quad \forall r\in(0,1/2).
 \end{align*}
  This  implies that our results recover   \cite[Theorem 11.1, Theorem 12.1]{CS1}.   
  In particular, considering the following  example of slowly varying functions at zero:  for   $l_0(r):= \left( \log\frac{2}{r^2}\right),$
 $$l_{0}^{\,\beta}(r )= \left( \log\frac{2}{r^2}\right)^{\beta} \quad\mbox{   $\forall r\in(0,1),$ $\beta\in\R$},$$    Theorem \ref{thm-Harnack-l-sigma-intro} asserts  the uniform Harnack inequality  and H\"older estimate of the elliptic  integro-differential operators associated with   the     regularly varying  kernel $K_{\sm,\beta}$ at zero with  index  $-\sm\in(-2,-\sm_0]$ for    $\beta\in\R$:  
 \begin{equation*}
(2-\sm)\f{ \ld}{|y|^{n+\sm}}\left( \log\frac{2}{|y|^2}\right)^{\beta(2-\sm)} \le
K_{\sm,\beta}(y)  \le  (2-\sm)\f{ \Ld}{|y|^{n+\sm}}\left( \log\frac{2}{|y|^2}\right)^{\beta(2-\sm)}\quad\mbox{near zero},
\end{equation*} 
 where the uniform constants  in the regularity estimates  depend only on $n,\ld,\Ld, \sm_0,\beta$ and the given slowly varying function $l_0$ at zero.    
 
 Lastly, we have the following theorem as  a corollary of Theorem \ref{thm-Holder-int-infty} by imposing \eqref{eq-integral-infty} instead of Property \ref{hypo-kernel-l-infty}.
   \begin{thm} \label{thm-Holder-l-sigma-int}
   Let $\sm_0\in(0,2),$ and  let  a measurable function $l_0:(0,+\infty)\to[0,+\infty)$ be  locally bounded away from $0$ and $+\infty$ on $(0,1]$, and    vary slowly  at zero such that $l_0(0)=1$.  For $\sm\in[\sm_0,2), $ define  $l_\sm:(0,+\infty)\to[0,+\infty)$ and $L_\sm:(0,1]\to(0,+\infty)$ as \eqref{eq-L-sigma}.  We assume that $$(2-\sm)\int_1^{\infty} \frac{l_\sm(s)}{s}ds \leq a_\infty\quad\mbox{for some  $a_\infty>0.$}$$  
 For $0<R<1,$ and $C_0>0,$ let   $u\in C(B_{2R})$ be a bounded   function in $\R^n$ such that  
 $$\cM^-_{  \fL\left(\ld,\Ld, l\right)} u\leq C_0 \quad\mbox{and}\quad\cM^+_{  \fL\left(\ld,\Ld, l\right)}u\geq -C_0\quad\mbox{in $B_{2R}$}$$ in the viscosity sense.  
Then 
we have 
$$ R^{\ap}\,[u]_{\ap, B_{R}}\leq C\left( \|u\|_{L^{\infty}(\R^n)}+\frac{C_0}{L(\rho_0 R)}\right), $$
 where     
   the  uniform constants $\ap\in(0,1),$     $C>0$ and $\rho_0\in(0,1)$ depend only $n,\ld,\Ld,   \sm_0, $  $a_\infty,$ and  the slowly varying function  $l_0$ at zero.  

\end{thm}


The rest of the paper is organized as follows. 
 Section \ref{sec-viscosity} contains an introduction to   viscosity solutions,   the ellipticity   for integro-differential operators and their properties.  
  Section \ref{sec-harnack} is devoted to the proof of  the uniform  regularity estimates 
 for a class  of fully nonlinear elliptic integro-differential  operators associated with  the kernels satisfying \eqref{eq-kernel2-intro-op} with  Properties \ref{hypo-kernel-l} and \ref{hypo-kernel-l-infty}. 
In  Section \ref{sec-uniform-harnack}, we prove  Proposition \ref{prop-l-sigma} and obtain   Theorems \ref{thm-Harnack-l-sigma-intro} and \ref{thm-Holder-l-sigma-int}  from Theorems \ref{thm-Harnack-regularly-varying-kernel-intro}, \ref{thm-Holder-intro} and  \ref{thm-Holder-int-infty}.   In Appendix \ref{sec-regular-variation},  we give the definitions of regularly and slowly varying functions and 
   summarize their  important  properties   which  are used in the paper.

\section{Viscosity solutions}\label{sec-viscosity}

In this section, we give  an   introduction to   the notions of    viscosity solutions and the ellipticity 
for    integro-differential operators as    in \cite{CS1}; see also \cite{LD1,KL2}. Important properties of viscosity solutions such as the stabilities under uniform convergence and the comparison principle are  also provided;  refer to \cite{CC} for the local case.       
We begin with  the concept of {\it    $C^{1,1}$ at the  point}.

\begin{definition}[$C^{1,1}$ at the point]\label{def-C^{1,1}}
Let $x\in\R^n.$  A function $\varphi$ is said to be $C^{1,1}$ at the point $x,$ denoted by $\varphi\in C^{1,1}(x),$ if there exist a vector $p\in\R^n$ and a number $M>0$ such that 
\begin{equation}\label{eq-C^{1,1}}
\left|\varphi(x+y)-\varphi(x)-p\cdot y\right|\leq M|y|^2\quad\mbox{for small $y\in\R^n$.}
\end{equation} For a set $\Omega\subset \R^n,$ we say  that $\varphi$ is $C^{1,1}$ in $\Omega$  when \eqref{eq-C^{1,1}} holds for any $x\in\Omega$ with a uniform constant $M>0.$ 
\end{definition}

 Now, we recall the   viscosity solutions for integro-differential operators. 
 \begin{definition}[Viscosity solution] \label{def-visc-sol} 
 Let     $\Omega\subset \R^n$ be an open set  
 and let $f$ be a   function in $\Omega. $
A  bounded  function  $u:  \R^n\to\R$ which is upper (lower) semi-continuous  in  $\overline\Omega$ is  called a viscosity subsolution (supersolution) to the integro-differential equation $\cI u=f$ in $\Omega$ and we write $\cI u \geq f$  in $\Omega $ ($\cI u \leq f$ in $\Omega$) when the following holds: if a $C^2$-function $\varphi$ touches $u$   from  above (below) at $x\in\Omega$ in a small neighborhood $N$ of $x, $ i.e., 
\begin{enumerate}[(i)]
\item $\varphi(x)=u(x),$
\item $\phi >u $ ($\phi<u$) in $N\setminus \{x\},$
\end{enumerate} 
then the function $v$ defined as 
\begin{equation*}
v:=\left\{  
\begin{split}
&\varphi\quad \mbox{in $N$,}\\
& u \quad\mbox{in $\R^n\setminus N,$}
\end{split}
\right.
\end{equation*}
satisfies $\cI v(x)\geq  f(x)$ ($\cI v(x)\leq  f(x)$). 
We say $u$ is  a viscosity solution if $u$ is  both a viscosity subsolution and a viscosity supersolution.
\end{definition}
Here, we consider {\it bounded}  viscosity solutions for nonlocal operators for simplicity. 
Our method  to prove the Harnack inequality for viscosity solutions  can   be  also employed under the assumption that 
the viscosity solutions have    a certain growth rate at infinity related to a    class of integro-differential operators to deal with; see   \cite{BI}  for  viscosity solutions  to integro-differential equations  in  a general framework.   
 
The  notion of ellipticity for integro-differential operators  is defined  making use of a nonlocal version of  the Pucci extremal  operators. 
For given    $0<\ld\leq\Ld <+\infty,$ and    a function $l:(0,+\infty)\to [0,+\infty)$ satisfying  Property \ref{hypo-kernel-l} and \eqref{eq-integral-infty}, which stays away from $0$ and $+\infty$ on $(0,1]$, 
let  
$\fL\left(\ld,\Ld, l\right)$
 denote the  class  of the following linear integro-differential  operators 
  with the    kernels $K$   satisfying  \eqref{eq-kernel2-intro-op}:
\begin{equation*}
\cL u(x)=\int_{\R^n}\mu(u,x,y)  K(y) dy,
\end{equation*}
where  $\mu(u,x,y):=u(x+y)+u(x-y)-2u(x).$ 
We recall the Pucci type  extremal operators:
\begin{equation*}
\begin{split}
\cM^+_{  \fL\left(\ld,\Ld, l\right)} u&:=\sup_{\cL\in\fL\left(\ld,\Ld, l\right)}\cL u,\quad \mbox{and}\quad\cM^-_{  \fL\left(\ld,\Ld, l\right)} u:=\inf_{\cL\in\fL\left(\ld,\Ld, l\right)}\cL u.
\end{split}
\end{equation*}   One can check that 
\begin{align*}
\cM^+_{  \fL\left(\ld,\Ld, l\right)} u(x)&= (2-\sm)\int_{\R^n}\left\{\Ld \mu^+(u,x,y) -\ld\mu^-(u,x,y)\right\} \frac{l(|y|)}{|y|^n}\,dy, \\
\cM^-_{  \fL\left(\ld,\Ld, l\right)} u(x)&= (2-\sm)\int_{\R^n}\left\{\ld \mu^+(u,x,y)^+ -\Ld\mu^-(u,x,y)\right\} \frac{l(|y|)}{|y|^n}\,dy,
\end{align*}
where   $\mu^{\pm}(u,x,y) :=\max \left\{\pm\mu(u,x,y),0\right\}. $ 

In terms of the Pucci type operators with respect to the class  $\fL(\ld,\Ld,l),$   the elliptic integro-differential operators with respect to $\fL(\ld,\Ld,l)$ are defined as below, which we can apply our results to. 

 \begin{definition}[Ellipticity for nonlocal operators] \label{def-elliptic-op}
 An operator $\cI$ is said to be  elliptic with respect to the class $\fL(\ld,\Ld, l)$ if  it satisfies the following.
 \begin{enumerate}[(i)]
 \item  If  a bounded function $u$ in $\R^n$ is of $C^{1,1}(x),$ then $\cI u(x)$ is defined      classically. 
 \item If  a bounded function $u$ in $\R^n$   is of $C^{1,1}(\Omega)$ for an open set $\Omega,$ then $\cI u(x)$ is continuous in $\Omega.$ 
 \item   For    bounded functions $u\in C^{1,1}(x) ,$ and $v\in C^{1,1}(x)$,   we have 
$$\cM^-_{\fL\left(\ld,\Ld, l\right)}   v(x) \leq \cI[u+v](x)-\cI u(x)\leq \cM^+_{\fL\left(\ld,\Ld, l\right)} v(x).$$  
\end{enumerate}
  \end{definition}

\begin{remark} 
{\rm
 If  a bounded function $u$ in $\R^n$ is   $C^{1,1}$ at the point $x,$ then the Pucci operators $\cM^{\pm}_{\fL(\ld,\Ld, l)}u(x)$ are  defined classically due to the properties (a) and (b) in Lemma \ref{lem-kernel-l-intro}. 
}
\end{remark}

In  Definition \ref{def-visc-sol},  a $C^2$-test function $\varphi$ can be  taken to be    $C^{1,1}$ only at the contact point $x$ for elliptic integro-differential operators.  We are led to   consider   a larger set of test functions and  a stronger concept of  the viscosity solution, however,  those approaches turn out to be  equivalent  thanks to the following lemma. 
 The proof  is   similar to      one of \cite[Lemma 4.3]{CS1} with the help of  
 Lemma \ref{lem-kernel-l-intro}; see also  \cite[Lemma 4.3]{LD1}. 
\begin{lemma}
 Let     $\Omega\subset \R^n$ be an open set and let  $\cI$ be an elliptic  integro-differential operator with respect to $\fL=\fL(\ld,\Ld,l)$.   Let   $u: \R^n\to\R$ satisfy $\cI u \geq f$ in $\Omega$ in the viscosity sense.   We assume  that a function  $\varphi  \in C^{1,1}(x)$ (for a point  $x\in\Omega$)    touches $u$ from above at $x$  in a small neighborhood $N$ of $x$. Then the function $v$ defined as 
\begin{equation*}
v:=\left\{  
\begin{split}
&\varphi\quad \mbox{in $N$,}\\
& u \quad\mbox{in $\R^n\setminus N,$}
\end{split}
\right.
\end{equation*}
 satisfies  $\cI v(x)\geq f(x)$ in the classical sense.
\end{lemma}
 
 Due to    Property \ref{hypo-kernel-l} and \eqref{eq-integral-infty} (Lemma \ref{lem-kernel-l-intro}),    the results of     \cite{CS1}  on  viscosity solutions for the elliptic  integro-differential operators  hold true for our  elliptic integro-differential operators    with respect to the class $\fL(\ld,\Ld,l)$.  
First,  the following lemma concerns  the nonlinear  integro-differential operators of the inf-sup type  \eqref{eq-Bellman}; the proofs can be found in \cite[Sections 3 and    4]{CS1}.

\begin{lemma}[Properties of the inf-sup type operators]\label{lem-bellman-property}
Let $\cI$ be the operator   in the form of  \eqref{eq-Bellman}.  Then we have the following. 
\begin{enumerate}[(a)]
\item   $\cI$ is an elliptic  integro-differential operator  with respect to $\fL(\ld,\Ld, l),$ that is,  $\cI$ satisfies: 
\begin{enumerate}[(i)]
\item For    bounded     functions $u $ and $v$ which are  $C^{1,1}$  at x, 
$$\cM^-_{\fL\left(\ld,\Ld, l\right)}   v(x) \leq \cI[u+v](x)-\cI u(x)\leq \cM^+_{\fL\left(\ld,\Ld, l\right)} v(x).$$  
\item If a bounded function  u in $\R^n$  is $C^{1,1}$ in an open set $\Omega,$ then $\cI u(x)$ is continuous in $\Omega.$ 
\end{enumerate}
\item If $u$ is a viscosity subsolution to $\cI u= f$ in an open set $\Omega,$ and a function $\varphi \in C^{1,1}(x)$  (for $x\in\Omega$) touches $u$ from above at $x$ in a small neighborhood of $x,$ then $\cI u(x)$ is defined  classically and $\cI u(x)\geq f(x).$  
\end{enumerate}
\end{lemma}

 
The  viscosity solutions to  the elliptic integro-differential equations   have nice    stability  properties  with respect to uniform convergence. Recalling the definition of $\Gamma$-convergence,   a slightly stronger     
  stability  of viscosity solutions  under $\Gamma$-convergence in Lemma \ref{lem-stab} is    quoted  from \cite[Lemma 4.5]{CS1}. 

%

 \begin{definition}[$\Gamma$-convergence]
 We say a  sequence of lower semi-continuous functions $u_k$    $\Gamma$-converges to $u$ in a set $\Omega\subset\R^n$ if it satisfies the  following conditions:
 \begin{enumerate}[(i)]
 \item For every sequence $x_k\to x$ in $\Omega$,  $$\displaystyle \liminf_{k\to\infty} u_k(x_k)\geq u(x).$$
 \item For every $x\in\Omega,$ there exists a sequence $x_k\to x$ in $\Omega$ such that 
 $$\limsup_{k\to\infty} u_k(x_k)=u(x).$$
 \end{enumerate} 
 \end{definition}

 \begin{lemma}[Stability]\label{lem-stab}
 Let $\cI$ be an elliptic operator with respect to the class $\fL(\ld,\Ld,l).$  For an open set $\Omega\subset \R^n,$  let $u_k$ be a sequence of  functions that are uniformly bounded in $\R^n$ 
 such that 
  \begin{enumerate}[(i)]
  \item $\cI u_k\leq f_k$ in $\Omega$  in the viscosity sense,
  \item $u_k\to u$ in the $\Gamma$-sense in $\Omega$,
  \item $u_k\to u$ a.e. in $\R^n$,
  \item $f_k \to f $ locally uniformly in $\Omega$ for some continuous function $f.$ 
  \end{enumerate}
  Then $\cI u\leq f$ in $\Omega$ in the viscosity sense. 
 \end{lemma}

 As a corollary, we   obtain the stability property under uniform convergence. 
 
 \begin{cor}
 Let $\cI$ be an elliptic operator with respect to the class $\fL(\ld,\Ld,l).$  For an open set $\Omega\subset \R^n,$  let $u_k\in C(\Omega)$ be a sequence of  functions that are uniformly bounded in $\R^n$   such that 
  \begin{enumerate}[(i)]
  \item $\cI u_k= f_k$ in $\Omega$ in the viscosity sense,
  \item $u_k\to u$ locally uniformly in $\Omega$,
  \item $u_k\to u$ a.e. in $\R^n$,
  \item $f_k \to f $ locally uniformly in $\Omega$ for some continuous function $f.$ 
  \end{enumerate}
  Then $\cI u= f$ in $\Omega$ in the viscosity sense. 
 \end{cor}


Lemma \ref{lem-comparison}  quoted from \cite[Theorem 5.9]{CS1} states that the difference of   two viscosity solutions  solves an equation in the same ellipticity class.  In the proof,  Jensen's approach  \cite{J}  using the inf- and sup-convolutions   was employed to compare two viscosity solutions to the   fully nonlinear elliptic integro-differential equations (see also \cite{A}); 
we refer to \cite[Chapter 5]{CC} for the local case. 
 
 \begin{lemma}\label{lem-comparison}  
 Let $\cI$ be an elliptic operator with respect to the class $\fL(\ld,\Ld,l).$  For an open set $\Omega\subset \R^n,$  let $u$ and $v$  be  bounded in $\R^n$   such that 
   $$\cI u\geq f,\quad\mbox{and} \quad \cI v\leq g\quad\mbox{in $\Omega$}$$ in the viscosity sense for two continuous functions $f$ and $g$.
  Then $$\cM^+_{\fL(\ld,\Ld,l)}(u-v)\geq f-g\quad\mbox{in $\Omega$}$$ in the viscosity sense. 
 \end{lemma}
 
 
 The comparison principle for the elliptic integro-differential operators  as  in \cite[Theorem 5.2]{CS1}  follows from  Lemma \ref{lem-comparison} with the help of  a barrier function given in Lemma \ref{lem-barrier-comparison};  see  also \cite[Assumption 5.1 and Lemma 5.10]{CS1}. 
 
 \begin{lemma}\label{lem-barrier-comparison}
 For a given $R\geq1,$ there exists $\delta_R>0$ such that the function $\varphi_R (x):=\min\left(1, \frac{|x|^2}{4R^2}\right)$ satisfies 
 $$\cM^-_{\fL(\ld,\Ld,l)}\varphi_R\geq \delta_R \quad\,\mbox{in $B_R$.} $$
 \end{lemma}
 \begin{proof}
 Let $x\in B_{R}.$ If $x\pm y\in B_{2R}, $ then we have $\mu\left(\varphi_R,x,y\right)= \frac{|y|^2}{2R^2}.$ If $x+y\not\in B_{2R},$ then $ \mu\left(\varphi_R,x,y\right)\geq 1- \frac{|x|^2}{2R^2}\geq \frac{1}{2}.$ Thus it follows that  for  $x\in B_R,$ 
 $$\cM^-\varphi_R(x)=(2-\sm)\ld \int_{\R^n} \mu\left(\varphi_R,x,y\right)\frac{l(|y|)}{|y|^n}dy\geq(2-\sm) \ld \int_{B_R}\frac{ |y|^2}{2R^2}\frac{l(|y|)}{|y|^n}dy
 =: \delta_R>0. 
 $$
 \end{proof}
 Lastly, we   state the comparison principle for  the fully nonlinear  elliptic integro-differential operators with respect to $\fL(\ld,\Ld,l)$; the proof is the same as    one for   Theorem 5.2 of \cite{CS1}. 
 \begin{thm}[Comparison principle]
 Let $\cI$ be an elliptic operator with respect to the class $\fL(\ld,\Ld,l).$  For a bounded  open  set $\Omega\subset \R^n,$  let $u$ and $v$  be  bounded in $\R^n$   such that 
  \begin{enumerate}[(i)]
   \item $\cI u\geq f$ and $\cI v\leq f$ in $\Omega$ in the viscosity sense for  some continuous functions $f$,
   \item $u\leq v$ in $\R^n\setminus \Omega$.
  \end{enumerate}
  Then     $u\leq v$ in $\Omega$.
 \end{thm}
 
\section{Regularity estimates for  integro-differential  
operators  with regularly varying kernels}\label{sec-harnack}

This section is  mainly devoted to  proving Theorems \ref{thm-Harnack-regularly-varying-kernel-intro} and \ref{thm-Holder-intro}, which   will provide the  Harnack inequality and H\"older estimate  for fully nonlinear   elliptic integro-differential operators associated with symmetric, regularly varying kernel at zero  and infinity 
as mentioned  in Remark \ref{rmk-Harnack}.   
Throughout  this   section,   let  $0<\ld \leq \Ld< +\infty,$ and let a measurable function $l:(0,+\infty)\to (0,+\infty)$  be  locally bounded away from $0$ and $+\infty$, and    satisfy Properties \ref{hypo-kernel-l} and \ref{hypo-kernel-l-infty} with the   positive constants $\sm\in[\sm_0,2)$,  $a_0\geq 1, a_\infty\geq 1,$  and $\rho\in(0,1)$  for a given      $\sm_0\in(0,2)$.   As in Subsection \ref{subsec-operators},  let $\fL\left(\ld,\Ld, l\right)$ be  the class of all linear   integro-differential operators  
\begin{equation*}
\cL u(x)=\int_{\R^n}\mu(u,x,y)K(y)\,dy
\end{equation*} 
 with  the  kernels  $K$     satisfying
 \begin{equation*}
   (2-\sm)\ld\f{ l (|y|)}{|y|^{n}}\le
K(y)\le  (2-\sm)\Ld\f{ l(|y|)}{|y|^{n}},
\end{equation*}
 where 
 $\mu(u,x,y):=u(x+y)+u(x-y)-2u(x).$ 
In order to prove   the  uniform   regularity estimates for  a class of viscosity solutions  to the  elliptic   integro-differential equations  with respect to the class  $\fL(\ld,\Ld,l), $    we will deal with  the Pucci type extremal operators $\cM^{\pm}_{\fL(\ld,\Ld, l)}$, 
  defined as  \eqref{eq-Pucci} and simply denoted by $\cM^{\pm}$, since the elliptic operator $\cI$
 satisfies   that 
   $$\cM^-_{\fL\left(\ld,\Ld, l\right)}   \leq \cI- \cI[0]\leq \cM^+_{\fL\left(\ld,\Ld, l\right)}.$$ 

 Before we proceed to   regularity estimates  for viscosity solutions to nonlocal equations, we study the     important  properties of  the given  function $l$ satisfying  Properties \ref{hypo-kernel-l}  and \ref{hypo-kernel-l-infty}, which will be used later. 
 \begin{lemma} \label{lem-kernel-l}
 Let a  measurable   function $l:(0,+\infty)\to (0,+\infty)$  be  locally bounded away from $0$ and $+\infty$, and satisfy  Properties \ref{hypo-kernel-l}  and \ref{hypo-kernel-l-infty} with   positive constants $\sm\in(0,2),$ $a_0\geq 1, a_\infty\geq 1,$  and $\rho\in(0,1).$ Then we have the following. 
 \begin{enumerate}[(a)]
 \item For $r\in(0,1],$    
$$   {\color{black}\frac{1}{2a_0}}\frac{ r^2l(r)}{2-\sm} \leq  \int_0^rsl(s)ds \leq {\color{black}2a_0}\frac{ r^2l(r)}{2-\sm}.  $$

   \item  For $   r \in(0,1],$
$$ \int_0^rs^3l(s)ds\leq {\color{black}a_0}r^4l(r).$$

 \item 
For $r\in(0,1],$
 $$L(r):=\sm\int_r^1\frac{l(s)}{s}ds \geq   {\color{black} \frac{1}{2a_0^2}}\left(r^{-{\color{black}\sm/2}}-1\right).$$
 In particular, for $\sm\in[\sm_0,2)$,   we have 
 $$   {\color{black} \frac{1}{2a_0^2}}\left(r^{-{\color{black}\sm_0/2}}-1\right) \leq L(r) \leq 2a_0^2 r^{-2} ,\quad\forall r\in(0,1].$$

\item
      $$ {\sm}\int_1^{\infty}\frac{l(s)}{s}ds\leq {\color{black}2a_\infty}.$$
\end{enumerate}
\end{lemma}
 
\begin{proof}
Using Property \ref{hypo-kernel-l},  we have that  for $r\in(0,1],$ 
\begin{align*}
\int_0^r sl(s)ds &= l(r)\int_0^rs\frac{l(s)}{l(r)} ds \leq a_0 r^{\sm+\delta}l(r)\int_0^rs^{1-\sm-\delta }   ds\\
 &\leq a_0 r^{\sm+\delta}l(r)\frac{1}{2-\sm-\delta} r^{2-\sm-\delta }  \leq \frac{2a_0}{2-\sm} r^{2}l(r)  
\end{align*}
since $\delta\in\left[0, \frac{1}{2}\min( 2-\sm ,\sm)\right).$   In a similar way,  we deduce  that  for $r\in(0,1]$
$$   {\color{black}\frac{1}{2a_0}}\frac{ r^2l(r)}{2-\sm} \leq  \int_0^rsl(s)ds \leq {\color{black}2a_0}\frac{ r^2l(r)}{2-\sm}$$
and $$\displaystyle  \int_0^rs^3l(s)ds\leq {\color{black}a_0}r^4l(r).$$
Recalling from Property \ref{hypo-kernel-l} that for $r\in(0,1]$  $$L(r):= {\sm}\int_{r}^1 \frac{l(s)}{s}ds,$$  it  follows that
\begin{align*}
L(r) &=  {\sm}l(r)\int_r^1\frac{1}{s}\frac{l(s)}{l(r)} ds \\
&\geq \frac{\sm }{a_0}r^{\sm+\delta}l(r)\int_r^1s^{-1-\sm-\delta }   ds
= \frac{\sm }{a_0}r^{\sm+\delta}l(r)\frac{1}{\sm+\delta}\left( r^{-\sm-\delta } -1\right)  \\
 &\geq \frac{1 }{2a_0}r^{\sm+\delta}\frac{l(1)}{a_0} r^{-\sm+\delta}\left( r^{-\sm-\delta } -1\right) = \frac{1 }{2a_0^2 } \left( r^{-\sm+\delta } - r^{2\delta}\right)  
 \\&\geq \frac{1 }{2a_0^2 } \left( r^{-\sm/2 } -1\right) . 
\end{align*}
 Arguing in a similar way,,   we have 
 $$    \frac{1 }{2a_0^2 } \left( r^{-\sm/2 } -1\right)  \leq L(r) \leq 2a_0^2 r^{-2} ,\quad\forall r\in(0,1].$$
  Since  $l(1)=1$,   Property  \ref{hypo-kernel-l-infty}  yields that  
      $${\sm}\int_1^{\infty}\frac{l(s)}{s}ds\leq \sm a_\infty  \int_1^\infty s^{-1-\sm+\delta'}ds \leq a_\infty \frac{\sm}{\sm-\delta'}\leq 2a_\infty,$$
completing the proof.

\end{proof}
  
\subsection{Aleksandrov-Bakelman-Pucci type estimate}

First, we extend the nonlocal Aleksandrov-Bakelman-Pucci(ABP)      estimate  by Caffarelli and Silvestre \cite[Lemma 8.1]{CS1}  for   fully nonlinear elliptic  integro-differential operators   with respect to $\fL(\ld,\Ld,l).$

\begin{lemma}\label{lem-abp-ring} 
Let $R\in(0,1)$ and       $\rho_0\in(0,1)$.  Let     $r_k:=\rho_0 2^{-\frac{1}{2(2-\sm)}-k}R,$    and  $\cR_k(x):=B_{r_k}(x)\setminus B_{r_{k+1}}(x)$ for $k\in\N\cup\{0\}$  and   $x\in\R^n.$
Let  $u$ be a viscosity subsolution of     
$$\cM^+_{\fL(\ld,\Ld,l)} u= -f\quad\mbox{ on $B_R$}$$   
such that $u\leq 0$  in $\R^n\setminus B_{R},$ 
and let  $\Gamma$ be the concave envelope of $u^+:=\max\{u,0\}$ in $B_{3R}.$    Then there exists a uniform constant $\displaystyle \tilde C :=   \frac{c_na_0}{\ld \rho_0^4} \sup_{\sm\in[\sm_0,2)}\left( \frac{1-2^{-2(2-\sm)}}{2-\sm}\right)  >0$  such that for each  $x\in \left\{ u=\Gamma \right\}$  and $M>0,$  we find  $k\in\N\cup\{0\}$  
 satisfying    
\begin{equation}\label{eq-abp-ring}
\left|  \left\{y\in \cR_k(x):  u(y)<u(x)+(y-x)\cdot\D\Gamma(x)-Mr_k^2\right\}\right|\leq \frac{\tilde C}{l(R)R^2}\frac{f(x)}{M}\left|\cR_k(x)\right|,
\end{equation}
where $\D\Gamma(x)$    stands for an element of  the superdifferential of $\Gamma$ at $x,$ and $c_n>0$ depends only on dimension $n.$
 \end{lemma}
\begin{proof} 
Let $x$ be a contact point, that is,  $x\in \{u=\Gamma \}\subset B_{R}.$ 
From Lemma \ref{lem-bellman-property}, $\cM^+u(x)$ can be  defined classically and   we have 
\begin{align*}
\cM^+ u(x)=(2-\sm)\int_{\R^n}\left\{\Lambda \mu ^+(u,x,z)-\lambda \mu ^-(u,x,z)\right\}\frac{l(|z|)}{|z|^n}dz\geq -f(x),
\end{align*}
where $\mu(u,x,z)=u(x+z)+u(x-z)-2u(x),$ and  $\mu ^{\pm}(u,x,z)=\max\{\pm \mu(u,x,z),0\}.$
 
 We note that $u(x)=\Gamma(x)>0.$ 
If $x+z\in B_{3R}$ and $x-z\in B_{3R}, $ then  we have $\mu(u,x,z)\leq0$ since $\Gamma $ lies above $u.$  If $x+z\not\in B_{3R},$  then  $x+z$ and $x-z$ do  not belong to $B_{R},$ which implies that $\mu(u,x,z)\leq0.$   Thus   it follows that   $\mu(u,x,z)\leq0$ for any $z\in\R^n$ and hence  
\begin{align*}
f(x)&\geq  (2-\sm)\lambda\int_{\R^n} \mu^-(u,x,z) \frac{l(|z|)}{|z|^n}dz\\
& \geq (2-\sm)\lambda\int_{B_{r_0}(0)} \mu ^-(u,x,z) \frac{l(|z|)}{|z|^n}dz =  (2-\sm) \lambda\sum_{k=0}^{+\infty}\int_{\cR_k(0)} \mu ^-(u,x,z) \frac{l(|z|)}{|z|^n}dz.
\end{align*}
The concavity of $\Gamma$ implies that if  $u(x+z)<u(x)+z \cdot\D\Gamma(x)-M r_k^2 $ for some $z\in B_{R}, $ then $\mu ^-(u,x,z)\geq M r_k^2.$ Indeed,  we have  that $x\pm z\in B_{3R} $ and  
\begin{align*}
\mu(u,x,z)&=u(x+z)+u(x-z)-2u(x)\leq u(x+z)+\Gamma(x-z)-2u(x)\\ 
&< \left\{u(x)+z \cdot\D\Gamma(x)-M r_k^2 \right\}+ \left\{\Gamma(x)-z\cdot\D \Gamma(x)\right\}-2u(x)=-M r_k^2.
\end{align*}

Suppose to the contrary  that \eqref{eq-abp-ring} is not true.   Then we have 
\begin{align*}
\frac{f(x)}{2-\sm}&\geq   \lambda\sum_{k=0}^{+\infty}\int_{\cR_k(0)} \mu ^-(u,x,z) \frac{l(|z|)}{|z|^n}dz=   \lambda l(R)\sum_{k= 0}^{+\infty}\int_{\cR_k(0)} \mu ^-(u,x,z) \frac{l(|z|)}{l(R)}\frac{1}{|z|^n}dz\\
&{\color{black}\geq}  \lambda l(R)\sum_{k= 0}^{+\infty}\int_{\cR_k(0)} \mu ^-(u,x,z) \frac{1}{a_0}\left(\frac{|z|}{R}\right)^{-\sm+ \delta}\frac{1}{r_k^n}dz\\
&{\color{black}\geq}  \lambda l(R)\frac{1}{a_0}  \sum_{k= 0}^{+\infty} \left(\frac{r_k}{R}\right)^{-\sm+ \delta }\frac{1}{r_k^n}\int_{\cR_k(0)} \mu ^-(u,x,z) dz\\
&{\color{black}\geq}  \lambda l(R)\frac{1}{a_0}  \sum_{k= 0}^{+\infty} \left(\frac{r_k}{R}\right)^{-\sm+ \delta }\frac{1}{r_k^n} \frac{\tilde C }{l(R)R^2}\frac{f(x)}{M}|{\cR_k(x)}| Mr_k^2\\
&{\color{black}=}  \lambda l(R)\frac{1}{a_0}  \sum_{k= 0}^{+\infty} \left(\frac{r_k}{R}\right)^{-\sm+ \delta }\frac{r_k^2}{r_k^n} \frac{\tilde C f(x)}{l(R)R^2}|{\cR_k(0)}| = \frac{\lambda c_n }{a_0}  \tilde C f(x) \sum_{k= 0}^{+\infty} \left(\frac{r_k}{R}\right)^{ 2-\sm+\delta} \\ 
&\geq  \frac{\lambda c_n}{ a_0}  \tilde C f(x) \sum_{k= 0}^{+\infty} \left(\frac{r_k}{R}\right)^{ 2-\sm+(2-\sm)/2}\geq  \frac{\lambda c_n}{ a_0}  \tilde C f(x) \sum_{k= 0}^{+\infty} \left(\frac{r_k}{R}\right)^{ 2(2-\sm)} \\ 
& =   \frac{ \lambda c_n}{ a_0}  \tilde C f(x) \frac{\rho_0^{2(2-\sm)}}{2\left(1-2^{-2(2-\sm)}\right)} \geq \frac{  \lambda c_n }{a_0}  \tilde C f(x) \frac{\rho_0^{ 4 }}{2\left(1-2^{-2(2-\sm)}\right)}
\end{align*}
since $0\leq \delta\leq (2-\sm)/2.$
By choosing $\displaystyle \tilde C\geq   a_0\frac{2 \left(1-2^{-2(2-\sm)}\right)}{\ld c_n  \rho_0^{4 }(2-\sm)}  $ which is   bounded above by a uniform constant  for $\sm\in[\sm_0,2),$  the result follows. 
 \end{proof}
In   the proof of  Lemma \ref{lem-abp-ring},   we observe  that  $f(x)$ is positive  for   $x\in \{u=\Gamma\}$.

 
\begin{lemma}\label{lem-grad-Gamma}
Under the same  assumption as  Lemma \ref{lem-abp-ring}, there exists uniform  constants $\ep_n\in(0,1)$ and $\displaystyle\tilde M:= {\tilde C}/{\ep_n}>0$ such that for  each $x\in\{u= \Gamma\},$ we find  some $r=r_k\leq \rho_0 2^{-\frac{1}{2(2-\sm)}}R$   which satisfies the following: 
\begin{enumerate}[(a)]
\item 
$$\left|  \left\{y\in B_r(x)\setminus B_{r/2}(x):  u(y)<u(x)+(y-x)\cdot\D\Gamma(x)-\frac{\tilde  M f(x)}{l(R)R^2}r^2\right\}\right|\leq \ep_n\left| B_r(x)\setminus B_{r/2}(x)\right|,$$
\item $$\Gamma(y)\geq u(x)+(y-x)\cdot\D\Gamma(x)-\frac{\tilde M f(x)}{l(R)R^2}r^2\quad\forall y\in B_{r/2}(x),$$
\item $$\left|\D\Gamma(B_{r/4}(x))\right|\leq  c_n\left(\frac{\tilde M f(x)}{l(R)R^2}\right)^n| B_{r/4}(x) |,$$
where $\tilde C>0 $ is the uniform  constant as in Lemma \ref{lem-abp-ring}, and  $\ep_n\in(0,1)$ and $c_n>0$ are  uniform constant depending only on $n.$
\end{enumerate}
\end{lemma}
 \begin{proof}
  For a small $\ep_n>0,$ let  $\tilde M :=\tilde C/\ep_n.$ 
 We apply Lemma \ref{lem-abp-ring} with $M=\tilde M\frac{f(x)}{l(R)R^2}$ to have 
\begin{align*}
\left|  \left\{y\in B_r(x)\setminus B_{r/2}(x):  u(y)<u(x)+(y-x)\cdot\D\Gamma(x)-\frac{\tilde M f(x)}{l(R)R^2}r^2\right\}\right|\leq \ep_n\left| B_r(x)\setminus B_{r/2}(x)\right|
\end{align*}
for some $r=r_k,$ which proves (a). 
 With the help of  (a), we employ the same arguments as  the proofs of  Lemma 8.4 and Corollary 8.5 in \cite{CS1} to show  (b) and (c). 
 \end{proof}

Now we obtain a nonlocal version of the ABP estimate in the following
theorem  making use of  Lemma \ref{lem-grad-Gamma} together with a dyadic cube decomposition; we refer  to \cite[Theorem 8.7]{CS1} for the proof. 
 \begin{thm}[ABP type estimate]\label{thm-abp}
 Let $R\in(0,1),$ and let      $\rho_0\in\left(0, 1/(32\sqrt{n})\right]$  be a constant.   Let  $u$ be a viscosity subsolution of     
$$\cM^+_{\fL(\ld,\Ld,l)} u= -f\quad\mbox{ on $B_R$}$$   
such that $u\leq 0$  in $\R^n\setminus B_{R},$ 
and let  $\Gamma$ be the concave envelope of $u^+$ in $B_{3R}.$  Then there exists a finite,  disjoint   family of open cubes $Q_j$ with diameters $d_j\leq  \rho_0 2^{-\frac{1}{2(2-\sm)}}R$ 
 such that $\left\{\overline Q_j\right\}$ covers the contact set $\{u=\Gamma\},$ and satisfies the following:
 \begin{enumerate}[(a)]
\item 
$\{u=\Gamma\}\cap \overline Q_j \not\eq\emptyset$ for any $Q_j,$
\item 
$\displaystyle |\D \Gamma(\overline Q_j)|\leq  c_n\left(\frac{\tilde C}{l(R)R^{2}}\right)^n\left(\max_{\overline Q_j \cap\{u=\Gamma\}}f^n\right) |Q_j|,$
\item
$\displaystyle\left|\left\{y\in 32\sqrt{n}Q_j : u(y)\geq\Gamma(y)-c_n \frac{\tilde  C}{l(R)R^2} \left(\max_{\overline Q_j\cap \{u=\Gamma\}}f \right)d_j^2\right\}\right|\geq \mu |Q_j|$
\end{enumerate}
for  $\mu:=1-\ep_n\in(0,1), $ where $\tilde C>0 $ is the uniform constant as in Lemma \ref{lem-abp-ring}, and  $\ep_n\in(0,1)$  (appearing  in  Lemma \ref{lem-grad-Gamma}) and $c_n>0$ are  uniform constant depending only on $n.$ 
\end{thm}

\subsection{Barrier function }

As in \cite{CS1} and \cite{KL2}, we  construct the  barrier function at each scale, where the monotone function $L$ associated with $l$ given in Property \ref{hypo-kernel-l} plays a role to obtain scale invariant estimates.  
\begin{lemma}\label{lem-barrier-pre}  
Let   $R\in(0,1/2).$   For  
 $\kappa_1\in(0,1),$ there exist uniform constants $p=p(n,\lambda,\Lambda)>n,$ and     $\ep_0 \in(0,1/8)$ such that  the function $\vp(x):=\min\left\{ |\kappa_0R|^{-p}, |x|^{-p}\right\}$ for $\kappa_0:=\ep_0\kappa_1>0$ satisfies 
 $$\cM^-_{\fL(\ld,\Ld,l)}\vp(x)\geq0, \quad\forall {x\in B_{R}\setminus B_{\kappa_1R}}.$$ 
\end{lemma}

\begin{proof}
Assume without loss of generality that  
    $x=R_0e_1$ for   $\kappa_1R\leq R_0<R.$  We  need to compute   
\begin{align*}
\cM^- \vp(x)&=  (2-\sm)\int_{\R^n}\left\{\lambda \mu^+(\vp,x,y)-\Lambda \mu^-(\vp,x,y)\right\}\frac{l(|y|)}{|y|^n}dy\\
&= (2-\sm) \int_{\R^n}\frac{\lambda \mu ^+}{2}\frac{l(|y|)}{|y|^n}dy+ (2-\sm)\int_{B_{\rho_1R_0}}\left(\frac{\lambda}{2} \mu ^+-\Lambda \mu ^-\right)\frac{l(|y|)}{|y|^n}dy 
\\
&+ (2-\sm) \int_{B_{\rho_1R_0}^c}\left(\frac{\lambda}{2} \mu ^+-\Lambda \mu ^-\right)\frac{l(|y|)}{|y|^n}dy\\
&\geq(2-\sm) \int_{\R^n}\frac{\lambda \mu ^+}{2}\frac{l(|y|)}{|y|^n}dy+ (2-\sm)\int_{B_{\rho_1R_0}}\left(\frac{\lambda}{2} \mu ^+-\Lambda \mu ^-\right)\frac{l(|y|)}{|y|^n}dy 
\\&- 2(2-\sm) \Lambda R_0^{-p}\int_{B_{\rho_1R_0}^c}\frac{l(|y|)}{|y|^n}dy=:I_1+I_2+I_3,
\end{align*}
where $\rho_1\leq \min\{\rho, 1/2\}$ will be chosen sufficiently small   later .

 For {\color{black}$|y|<\frac{R_0}{2},$}  we have 
\begin{equation}\label{eq-2nd-dif-lower}
\begin{split}
|x+y|^{-p}+|x-y|^{-p}-2|x|^{-p}&= R_0^{-p}\left\{\left|\frac{x}{R_0}+\frac{y}{R_0}\right|^{-p}+\left|\frac{x}{R_0}-\frac{y}{R_0}\right|^{-p}-2\right\}\\
&\geq R_0^{-p}p\left\{-|\overline{y}|^2+(p+2)\overline{y}_1^{2}-\frac{1}{2}(p+2)(p+4)\overline{y}_1^2|\overline{y}|^2\right\}
\end{split}
\end{equation}
for $\overline y:=y/R_0$; see \cite[Lemma 9.1]{CS1}.    
We choose {\color{black}$\N\owns p>n$} large enough so that 
\begin{equation}\label{eq-barrier-choose-p}
(p+2)\frac{\lambda}{2}\int_{\pa B_1}y_1^2d\sm(y)-\Lambda|\pa B_1|=:\overline\delta>0.
\end{equation}
We use \eqref{eq-2nd-dif-lower}, \eqref{eq-barrier-choose-p}, Lemma \ref{lem-kernel-l}  and  Property \ref{hypo-kernel-l}   to obtain 
\begin{align*}
I_2&= (2-\sm)\int_{B_{\rho_1R_0}}\left(\frac{\lambda}{2}\mu^+-\Lambda \mu ^-\right)\frac{l(|y|)}{|y|^n}dy\\
&\geq(2-\sm) pR_0^{-p}  \int_{B_{\rho_1R_0}}\left\{ 
\frac{\lambda}{2}(p+2) \frac{y_1^2}{R_0^2}- \Lambda\left(\frac{|y|^2}{R_0^2}+\frac{(p+2)(p+4)}{2}  \frac{y_1^2|y|^2 }{R_0^4}\right)\right\} \frac{l(|y|)}{|y|^n}dy\\
&\geq (2-\sm)pR_0^{-p}  \left\{
\frac{\overline\delta}{R_0^2}\int_0^{\rho_1R_0}sl(s)ds  -\frac{\Ld(p+2)(p+4)\omega_n}{2R_0^4}
\int_0^{\rho_1R_0}  s^3l(s)ds\right\}\\
 &   \geq   (2-\sm)pR_0^{-p} \left\{
 \frac{ \overline\delta }{2a_0(2-\sm)} \rho_1^2l(\rho_1R_0)
  - \frac{\Ld(p+2)(p+4)\omega_n}{2}a_0 \rho_1^4
l(\rho_1R_0) \right\}\\
&=   {pR_0^{-p}} \left\{ \frac{\overline\delta  }{2a_0}\rho_1^2 - (2-\sm)\frac{\Ld(p+2)(p+4)\omega_n }{2}a_0 \rho_1^4\right\} l(\rho_1R_0) \\
&\geq \frac{pR_0^{-p}}{2} \left\{
\frac{ \overline\delta}{ 2a_0} \rho_1^2
  - (2-\sm)\frac{\Ld(p+2)(p+4)\omega_n}{2}a_0  
\right\} L(\rho_1R_0). 
\end{align*}
We  select a uniform constant  $\rho_1=\rho_1(a_0, a_\infty,\sm_0) \leq\min(\rho,1/2)$ small so that 
\begin{equation}\label{eq-choice-rho1}
2a_\infty \leq \frac{1}{2a_0^2}\left(\rho_1^{-{\sm_0/2}}-1\right)\leq L(\rho_1),
\end{equation}
 and  then we use    Lemma \ref{lem-kernel-l} and  Properties \ref{hypo-kernel-l}   and \ref{hypo-kernel-l-infty}  again to have
\begin{align*}
-I_3&=2(2-\sm) \Lambda R_0^{-p}\int_{B_{\rho_1R_0}^c}\frac{l(|y|)}{|y|^n}dy\leq 2(2-\sm) \Ld R_0^{-p}  \omega_n\left\{ \frac{2a_\infty }{\sm}+ \frac{1}{\sm}L(\rho_1 R_0)\right\}\\
&\leq \frac{2-\sm}{\sm}2 \Ld \omega_nR_0^{-p} \left\{    \frac{1}{2a_0^2}\left(\rho_1^{-{\sm_0/2}}-1\right)+ L(\rho_1 R_0)\right\} 
\leq \frac{2-\sm}{\sm_0}4 \Ld \omega_nR_0^{-p} L(\rho_1 R_0),
\end{align*}
where we note that $L$ is monotone. 
 Thus  we deduce    that
\begin{align*}
 I_2+I_3 &\geq \frac{pR_0^{-p}}{2} \left\{
 \frac{\overline\delta}{2 a_0}\rho_1^2 -  (2-\sm) \frac{\Ld(p+2)(p+4)\omega_na_0}{2} 
\right\} L(\rho_1R_0)  -\frac{2-\sm}{\sm_0}4\Ld \omega_nR_0^{-p} L(\rho_1 R_0)\\
&=\frac{R_0^{-p}}{2} L(\rho_1 R_0)   \left\{
 \frac{p\overline\delta}{2 a_0}\rho_1^2
  -  (2-\sm) \frac{\Ld p(p+2)(p+4)\omega_na_0}{2} - (2-\sm) \frac{8\Ld \omega_n}{\sm_0}\right\} \\
 & \geq0 
\end{align*}
 for any $\sm\in[\sm_1,2), $
where $\sm_1\in[\sm_0,2)$ depends only on $n,\ld,\Ld, a_0,a_\infty,\rho,$ and $ \sm_0.$ 
Thus  the lemma holds true for $\sm\in[\sm_1,2).$

For $\sm\in [\sm_0,\sm_1)$, we  will make  $I_1$ sufficiently  large   by selecting $\kappa_0>0$ small.  For $x=R_0e_1$ with $\kappa_1R\leq R_0<R,$ we have that for  $\kappa_0:=\ep_0\kappa_1\in(0,\kappa_1/8)$
\begin{align*}
  I_1&\geq  (2-\sm_1) \int_{\R^n}\frac{\lambda\mu^+}{2}\frac{l(|y|)}{|y|^n}dy\\
&\geq    (2-\sm_1)  \frac{ \ld}{2}\int_{B_{R_0/4}(x)}\left\{|x-y|^{-p}-2R_0^{-p}\right\}\frac{l(|y|)}{|y|^n}dy\\
&\geq    (2-\sm_1) \frac{ \ld}{4}\int_{B_{R_0/4}(x)\setminus B_{\kappa_0R}(x)} |x-y|^{-p} \frac{l(|y|)}{|y|^n}dy=    (2-\sm_1) \frac{ \ld}{4}\int_{B_{R_0/4}(0)\setminus B_{\kappa_0R}(0)} |z|^{-p} \frac{l(|x+z|)}{|x+z|^n}dz\\
&\geq      (2-\sm_1)\frac{ \ld}{2^{n+2}R_0^{n}}\left(\min_{s\in[R_0/2,3R_0/2]}l(s)\right)\omega_n\int^{{R_0/4}}_{\kappa_0R} s^{-p+n-1} ds\\
&\geq       (2-\sm_1) \frac{ \ld}{2^{n+2}R_0^{n}}\frac{1}{p-n}\left\{(\kappa_0R)^{-p+n}-(R_0/4)^{-p+n}\right\}
\min_{s\in[R_0/2,3R_0/2]}l(s)\\
&\geq    (2-\sm_1) \frac{ \ld}{2^{n+2}R_0^{n}}\frac{R_0^{-p+n}}{p-n}\left\{\left(\frac{\kappa_1}{\kappa_0}\right)^{p-n}-4^{p-n}\right\}
\min_{s\in[R_0/2,3R_0/2]}l(s)\\
&\geq  (2-\sm_1)  \frac{ \ld}{2^{n+3}R_0^{n}}\frac{R_0^{-p+n}}{p-n}\left(\frac{\kappa_1}{\kappa_0}\right)^{p-n}
\min_{s\in[R_0/2,3R_0/2]}l(s)\\
& \geq  (2-\sm_1)  \frac{ \ld}{2^{n+3}R_0^{n}}\frac{R_0^{-p+n}}{p-n}\ep_0^{-p+n}
\frac{1}{a_0}\left(\frac{2\rho_1}{3}\right)^{\sm+\delta}l(\rho_1 R_0)\\
& \geq   (2-\sm_1)  \frac{ \ld}{2^{n+3}}\frac{R_0^{-p}}{p-n} \frac{1}{\ep_0}
\frac{1}{a_0}\left(\frac{2\rho_1}{3}\right)^{3 } {\color{black} }\frac{1}{2} L(\rho_1 R_0).
\end{align*}
From the   argument above, we notice that for $\sm\in[\sm_0,2),$ 
$$ I_2+I_3 \geq -C R_0^{-p}L(\rho_1R_0),$$
where  a uniform constant $C>0$ depends only on $n,\ld,\Ld, a_0,a_\infty,\rho,$ and $ \sm_0.$  Therefore,  we choose a uniform constant $\ep_0=\frac{\kappa_0}{\kappa_1}\in(0,1/8)$ sufficiently  small   to conclude that $$\cM^-\vp(x)\geq I_1+I_2+I_3\geq0$$
for $x\in B_{R}\setminus B_{\kappa_1 R}$ in the case when $\sm\in[\sm_0,\sm_1).$  This finishes the proof.
\end{proof}

\begin{lemma}\label{lem-barrier}
Let $R\in(0,1/2)$ and $0<\delta_1<\delta_2<1.$ Assume $\delta_1\leq \rho_1, $ where  $\rho_1=\rho_1(a_0, a_\infty, \rho,\sm_0) \leq\min(\rho,1/2)$ is the constant satisfying \eqref{eq-choice-rho1}.   There exists a continuous function $\Phi$ in $\R^n $     such  that
\begin{enumerate}[(a)]
\item $\Phi$ is nonnegative and uniformly bounded in $\R^n$,
\item $\Phi= 0$ outside $B_R,$
\item $\Phi\geq 2$  in $B_{\delta_2R},$
\item {  $L(\delta_1 R)^{-1}\cM^-_{\fL(\ld,\Ld,l)}\Phi\geq -\psi\ $}  in $\R^n $
for    some nonnegative, uniformly bounded function $\psi$ such that   $\supp(\psi)\subset B_{\delta_1 R}.$ 
\end{enumerate}
\end{lemma}
\begin{proof}
Let $\kappa_1:=\delta_1/2.$ According to Lemma \ref{lem-barrier-pre}, 
    the  function $\vp(x):=\min\left\{ |\kappa_0R|^{-p}, |x|^{-p}\right\}$ satisfies 
 $$\cM^-\vp(x)\geq0 \quad\forall {x\in B_{R}\setminus B_{\kappa_1R}}$$
 for some $0<\kappa_0<\kappa_1/8$ and $p>n.$ Now we  define $\Phi:\R^n\to [0,+\infty)$ by  
\begin{equation*}
\Phi(x):=c_0\left\{
\begin{split}
&P(x)\quad&\forall x\in B_{\kappa_0R}\\
&(\kappa_0R)^p\left\{ \min\left(|\kappa_0R|^{-p}, |x|^{-p}\right)-R^{-p}\right\}\quad&\forall x\in B_{R}\setminus B_{\kappa_0R}\\
& 0& \mbox{outside $B_R$},
\end{split}
\right.
\end{equation*}
where  $P(x):=-a|x|^2+b$  with $a=\frac{1}{2}p (\kappa_0R)^{-2}$ and $b:=1-\kappa_0^p+\frac{1}{2}p.$ 
Thus $\Phi$ is a  $C^{1,1}$-function on $B_R.$ By setting $c_0:=\frac{2}{\kappa_0^p(\delta_2^{-p}-1)},$ the property (b) follows.  
  Note that Lemma \ref{lem-barrier-pre} implies that 
 $$\cM^-\Phi(x)\geq0 \quad\forall {x\in B_{R}\setminus B_{\delta_1R/2}}.$$
 
   It remains to show that   
$$  L(\delta_1R)^{-1}\cM^-\Phi\geq -C\quad\mbox{in $B_{\delta_1R}$}.$$ We use  Properties \ref{hypo-kernel-l} and \ref{hypo-kernel-l-infty} and Lemma \ref{lem-kernel-l} to deduce  
that  for $x\in B_{\delta_1R},$ 
\begin{align*}
\cM^-\Phi(x)&\geq- (2-\sm)\Ld\int_{\R^n}\mu^-(\Phi,x,y)\frac{l(|y|)}{|y|^n}dy\\
&\geq-  (2-\sm)\Ld\int_{B_{\delta_1R}}\mu^-(\Phi,x,y)\frac{l(|y|)}{|y|^n}dy- (2-\sm)\Ld 2c_0 b\int_{\R^n\setminus B_{\delta_1R}} \frac{l(|y|)}{|y|^n}dy\\
&\geq-  (2-\sm)\Ld\int_{B_{\delta_1R}}\mu^-(\Phi,x,y)\frac{l(|y|)}{|y|^n}dy-\Ld 4c_0 b\omega_n\frac{1}{\sm}\left\{ L(\delta_1R)+2a_\infty\right\}\\
&\geq - (2-\sm)\Ld c_1R^{-2}\omega_n {2a_0}\frac{\delta_1^2R^2}{2-\sm}l(\delta_1R)-\frac{8bc_0\Ld \omega_n }{\sm_0}L(\delta_1R)\\
&\geq -\Ld \omega_n\left\{   {4a_0c_1\delta_1^2} +\frac{8bc_0}{\sm_0} \right\}L(\delta_1R)
\end{align*}
since $D^2\Phi\geq  -c_1R^{-2}{\bf{I}}\,\,$ a.e. in $B_{2\delta_1R}$ for some $c_1=c_1(\delta_1,\delta_2)>0,$ 
where we recall that $0<\delta_1\leq\rho_1,$ and  $\rho_1=\rho_1(a_0, a_\infty, \rho,\sm_0) \leq\min(\rho,1/2)$ satisfies  \eqref{eq-choice-rho1}. 
\end{proof}

\subsection{Power decay estimate of super-level sets }

We use   the ABP type estimate in Theorem \ref{thm-abp} and the  barrier functions constructed in Lemma \ref{lem-barrier} to obtain   the measure  estimates of  super-level sets of the viscosity supersolutions to fully nonlinear elliptic integro-differential operators with respect to $\fL(\ld,\Ld,l)$.

\begin{lemma}\label{lem-decay-1st-step}
%
Let $0<R<1/2$ and let  $Q_r=Q_r(0)$ denote  a dyadic cube of side $r$ centered at $0$ for $r>0.$  There exist uniform constants $\vep_0,\, \rho_0, \,\mu_0\in(0,1)$ and $M_0>1$, depending only on $n,\ld,\Ld, a_0,a_\infty,\rho,$ and $ \sm_0,$  such that if 
\begin{enumerate}[(a)]
\item $u\geq0$ in $\R^n,$
\item $\displaystyle\inf_{Q_{ \frac{3R}{2\sqrt{n}}}}u\leq 1,$
\item $\cM^-_{\fL(\ld,\Ld,l)}u\leq \vep_0 L(\rho_0 R)$ on $Q_{2R}$ in the viscosity sense,
\end{enumerate}
then 
$$ {\left|\left\{u\leq M_0\right\}\cap Q_{ \frac{R}{2\sqrt{n}}}\right|}> \mu_0\left| Q_{ \frac{R}{2\sqrt{n}}}\right|.$$
\end{lemma}

\begin{proof} Let $\rho_0\in(0,1)$ be a constant to be chosen later depending only on $n$ and $\rho_1>0,$ where the constant  $\rho_1$ satisfies \eqref{eq-choice-rho1}.  Let $\Phi$  be the barrier function      in Lemma \ref{lem-barrier}  
  with  $\delta_1:=\rho_0 \left(\leq \min\left\{ {1}/{(32\sqrt{n})}, \rho_1\right\}\right),$ and $ \delta_2:= \frac{3}{4}.$  
Then 
$$v:=\Phi-u$$  satisfies that $v\leq 0$ outside $B_R,$ $\displaystyle\max_{B_R} v\geq 1$ and 
$$\cM^+_{\fL(\ld,\Ld,l)}v\geq \cM^-_{\fL(\ld,\Ld,l)}\Phi-\cM^-_{\fL(\ld,\Ld,l)}u\geq - (\psi+\vep_0) L(\rho_0R) 
 \quad\mbox{in $B_{R}$}$$
 in the viscosity sense.  
 For  the concave envelope $\Gamma$ of $v^+$ in $B_{3R},$   Theorem \ref{thm-abp}  with the help of  Property \ref{hypo-kernel-l} yields that  
\begin{align*}
\frac{1}{R}\leq \frac{1}{R}\max_{B_R}v&\leq c_n|\D \Gamma(B_R)|^{1/n}\leq c_n\left(\sum_j|\D\Gamma(\overline Q_j)|\right)^{1/n} \\
&\leq  c_n \frac{\tilde C}{l(R)R^{2}} {L(\rho_0R)}  \left(\sum_j\max_{\overline Q_j}\left(\psi+\vep_0\right)^n |Q_j|\right)^{1/n}\\
&\leq  c_n\tilde C \frac{L(\rho_0R)}{l(\rho_0R) R^2}  {a_0}{\rho_0^{-\sm-\delta}} \left(\sum_j\max_{\overline Q_j}\left(\psi^n+\vep_0^n\right) |Q_j|\right)^{1/n}\\
&\leq  c_n\tilde C \frac{a_0}  {\  \rho_0^{2 }} \frac{1}{R^{2}} \left(\sum_j\max_{\overline Q_j}\left(\psi^n+\vep_0^n\right) |Q_j|\right)^{1/n},
\end{align*}
so we have 
  \begin{align*}
\frac{1}{R}\leq  \displaystyle\frac{C}{R^{2}} \left(\sum_j\max_{\overline Q_j}\left(\psi^n+\vep_0^n\right) |Q_j|\right)^{1/n}
\end{align*}
for a uniform constant $C>0$ depending only on $n,\ld, a_0,\rho,$ and $ \sm_0,$
Recalling that  the nonnegative function $\psi$ is uniformly   bounded with  $\supp\,\psi\subset B_{\rho_0R}$ in Lemma \ref{lem-barrier}, and  $\sum_j|{Q_j}| \leq c_n|B_R|$,  it follows that  
\begin{align*}
\frac{1}{R}&\leq    \frac{C\vep_0}{R}+  \frac{C}{R^{2}}\left(\sum_{\overline Q_j\cap B_{\rho_0R}\not\eq\emptyset}|Q_j|\right)^{1/n}.
\end{align*}
 By  selecting  $\vep_0>0$ small, we have  
\begin{equation}\label{eq-decay-1st-Qj}
 \frac{C}{R}\left(\sum_{\overline Q_j\cap B_{\rho_0R}\not\eq\emptyset}|Q_j|\right)^{1/n}\geq \frac{1}{2}.
 \end{equation} 
   Now we select $\rho_0>0$ sufficiently small such that  $32\sqrt{n}\rho_0\leq \frac{1}{8\sqrt{n}}$ in order  to show that  $32\sqrt{n}Q_j\subset B_{\frac{R}{4\sqrt{n}}}$ 
for any $Q_j$ satisfying     $\overline Q_j\cap B_{\rho_0R}\not\eq\emptyset.$  
Thus      $\displaystyle \bigcup_{\overline Q_j\cap B_{\rho_0R}\not\eq\emptyset}\overline Q_j$  is covered by  $\left\{32\sqrt{n}Q_j\,:\,  \overline Q_j\cap B_{\rho_0R}\not\eq\emptyset\right\}$   contained in $ B_{\frac{R}{4\sqrt{n}}}.$

 On the other hand, according to Theorem \ref{thm-abp} together  with  the previous argument,  we have    
\begin{equation}\label{eq-w-Harnack-1st-Q_j-measure}
\displaystyle\left|\left\{y\in 32\sqrt{n}Q_j : u(y)\leq M_0 \right\}\right|\geq \mu |Q_j|
\end{equation}
for some $M_0>1.$     Indeed,    from the previous argument,  we see  that $$c_n \frac{\tilde C}{l(R)R^{2}} {L(\rho_0R)} \max_{\overline Q_j}\left(\psi+\vep_0\right) d_j^2\leq M_1  $$
for a uniform constant $M_1>1$ with respect to $\sm\in[\sm_0,2)$ 
 since $d_j\leq \rho_0R.$    Then   it follows from Theorem \ref{thm-abp}   that 
 \begin{align*}
  \mu |Q_j| &\leq  \left|\left\{y\in 32\sqrt{n}Q_j : v(y)\geq\Gamma(y)- M_1 \right\}\right|\\
  & \leq \left|\left\{y\in 32\sqrt{n}Q_j : u(y)\leq \|\Phi\|_{L^{\infty}(\R^n)}+ M_1=:M_0 \right\}\right| 
 \end{align*}
 since $\Phi$ is uniformly  bounded  in $\R^n,$ and $\Gamma$ is positive  in $B_{3R}$.
Taking a subcover of $\left\{32\sqrt{n}Q_j\,:\,  \overline Q_j\cap B_{\rho_0R}\not\eq\emptyset\right\}$   with finite overlapping,   we deduce from   \eqref{eq-decay-1st-Qj} and \eqref{eq-w-Harnack-1st-Q_j-measure} that  for uniform constants $M_0>1$ and $\mu_0\in(0,1),$    
$$
{\left|\left\{u\leq M_0\right\}\cap Q_{\frac{R}{2\sqrt{n}}}\right|}\geq \displaystyle\left|\left\{ u\leq M_0 \right\}\cap B_{\frac{R}{4\sqrt{n}}}\right| > \mu_0\left| Q_{\frac{R}{2\sqrt{n}}}\right|,$$
    which finishes the proof. 
\end{proof}

The Calder\'on-Zygmund technique  combined with Lemma \ref{lem-decay-1st-step}  implies the following decay measure  estimate of   super-level sets  making use of the monotonicity of  the function $L.$ 
\begin{cor} Under the same assumption as Lemma \ref{lem-decay-1st-step}, 
we have  $$ {\left|\left\{u> M_0^k\right\}\cap Q_{\frac{R}{2\sqrt{n}}}\right|}\leq (1-\mu_0)^k\left| Q_{\frac{R}{2\sqrt{n}}}\right|,\quad\forall k=1,2,\cdots,$$
and hence 
\begin{equation*}
\left|\{u>t\}\cap Q_{\frac{R}{2\sqrt{n}}}\right|\leq CR^n  t^{-\ep}\quad\forall t>0,
\end{equation*}
where $C>0$ and $\ep>0$ are uniform constants depending only on $n,\ld,\Ld, a_0,a_\infty,\rho,$ and $ \sm_0.$  
\end{cor}

Using a standard covering argument,  we deduce the weak Harnack inequality as follows. 
\begin{thm}[Weak Harnack inequality]\label{thm-weak-Harnack}
For $0<R<1,$ and $C_0>0,$  let $u $ be a nonnegative function  in $\R^n$    such that      
 $$\cM^-_{\fL(\ld,\Ld,l)} u\leq C_0  \quad\mbox{in $B_{2R}$}$$ in the viscosity sense.  
 Then we have 
 \begin{equation*}
\left|\{u>t\}\cap B_{R}\right|\leq CR^n\left( u(0)+ \frac{C_0}{L(\rho_0 R)}\right)^\ep t^{-\ep}\quad\forall t>0, 
\end{equation*}
and hence 
$$\left(\fint_{B_{R}}|u|^p\right)^{1/p}\leq C\left\{  u(0)+  \frac{C_0}{L(\rho_0 R)} \right\},$$
 where 
   $C>0 , \ep>0,  \rho_0\in(0,1),$ and $p>0$ are  uniform constants depending only on  $n,\ld,\Ld,   $  $ a_0,a_\infty,$ $ \rho$ and $\sm_0$. 
\end{thm}

\subsection{Harnack inequality}

Making use of the weak Harnack inequality  in Theorem \ref{thm-weak-Harnack}, we prove the scale invariant  Harnack inequality for fully nonlinear elliptic integro-differential operators with respect to $\fL(\ld,\Ld,l)$, where the constant in the Harnack estimate depends only on  on  $n,\ld,\Ld,   $  $ a_0,a_\infty,$ $ \rho$ (in Properties \ref{hypo-kernel-l} and \ref{hypo-kernel-l-infty}) and $\sm_0 $.  The proof of \cite[Theorem 11.1]{CS1} has been adapted to our elliptic  integro-differential operators associated with regularly varying kernels at zero  and  infinity.

\begin{thm}\label{thm-Harnack-regularly-varying-kernel}
For $0<R<1,$ and $C_0>0,$ let $u\in C(B_{2R})$ be a nonnegative   function in $\R^n$ such that  
 $$\cM^-_{\fL(\ld,\Ld,l)} u\leq C_0L(\rho_0 R), \quad\mbox{and}\quad\cM^+_{\fL(\ld,\Ld,l)}u\geq -C_0L(\rho_0 R)\quad\mbox{in $B_{2R}$}$$ in the viscosity sense, where $\rho_0\in(0,1)$ is the constant as in Theorem \ref{thm-weak-Harnack}. 
Then we have 
$$\sup_{B_{\frac{R}{2}}}  u\leq C\left( u(0)+  {C_0}\right),$$
 where  a uniform constant  $C>0$ depends only on  $n,\ld,\Ld,   $  $ a_0,a_\infty,$ $ \rho$ and $\sm_0$. 

\end{thm}
\begin{proof}
We may assume that $u>0, $ $u(0)\leq1,$ 
 and $C_0=1.$ Let $\ep>0$ be the constant as in Theorem \ref{thm-weak-Harnack} and let    $\gamma:=(n+2)/\ep.$ Consider the minimal value of $\ap>0$ such that
$$u(x)\leq h_\ap(x):=\ap \left(1-\frac{|x|}{R}\right)^{-\gamma}\quad\forall x\in B_R.$$ 
We claim that  $\ap>0$ is uniformly bounded. 
 Let $x_0$ be a point such that $u(x_0)=h_\ap(x_0). $ We may assume that $x_0\in B_R,$ otherwise $\ap$ is small. Let $ d:= R- |x_0|$ and $r:=d/2.$
 
 Let $A:=\left\{u> u(x_0)/2\right\}.$ According to the weak Harnack inequality in Theorem \ref{thm-weak-Harnack}, we have 
 \begin{equation*}
 |A\cap B_R|\leq CR^n \left(\frac{2}{u(x_0)}\right)^{\ep}\leq  CR^n \ap^{-\ep}\left(\frac{d}{R}\right)^{\gamma\ep}=  C  \ap^{-\ep}\left(\frac{d}{R}\right)^{3} d^n  \leq C  \ap^{-\ep} d^n.
 \end{equation*}
This implies that 
\begin{equation} \label{eq-harnack-pf-2}
 \left|\left\{u> u(x_0)/2\right\}\cap B_r(x_0)\right|    \leq C  \ap^{-\ep} |B_r(x_0)|
\end{equation}
since $B_r(x_0)\Subset B_R$ and $r=d/2.$

Now we will show that there is a uniform number $\theta\in(0,1)$ such that 
$$ \left|\left\{u< u(x_0)/2\right\}\cap B_{\theta r}(x_0)\right|    \leq \frac{1}{2}|B_{\theta r}(x_0)|$$ for a     large constant    $\ap>1,$ from  which \eqref{eq-harnack-pf-2}       yields  that $\ap>0$ is uniformly bounded. 
  We first  notice  that for $x\in B_{\theta r}(x_0),$
$$u(x)\leq h_\ap(x)\leq \ap \left(\frac{d-\theta r}{R}\right)^{-\gamma}=  \ap \left(\frac{d}{R}\right)^{-\gamma} \left(1-\frac{\theta }{2}\right)^{-\gamma}=\left(1-\frac{\theta}{2}\right)^{-\gamma}u(x_0).  $$
  For  $\theta\in(0,1),$ consider
$$v(x):=\left(1-\frac{\theta}{2}\right)^{-\gamma}u(x_0)-u(x).$$
  Note that   $v$ is nonnegative in $B_{\theta r}(x_0). $ To apply the weak Harnack inequality to $$w:=v^+,$$
 we will    compute $\cM^-w$  in $B_{\theta r}(x_0).$  
    First,  we see that for $x\in B_{\theta r}(x_0),$ 
\begin{align*}
\cM^- w(x)&= (2-\sm)\int_{\R^n} \left\{\ld \mu^+(v^+,x,y)-\Ld \mu^-(v^+,x,y)\right\}\frac{l(|y|)}{|y|^n} dy\\ 
&\leq \cM^-v(x)+ (2-\sm)\int_{\R^n} \left\{\Ld  v^-(x+y)+\Ld v^-(x-y)\right\}\frac{l(|y|)}{|y|^n} dy\\ 
&\leq  L(\rho_0 R)+ (2-\sm)\int_{\R^n} \left\{\Ld  v^-(x+y)+\Ld v^-(x-y)\right\}\frac{l(|y|)}{|y|^n} dy\\ 
&=  L(\rho_0 R)+ 2(2-\sm)\int_{\left\{v(x+y)<0\right\}} -\Ld  v^-(x+y) \frac{l(|y|)}{|y|^n} dy\\ 
&\leq   L(\rho_0 R)+2(2-\sm) \Ld\int_{\R^n\setminus B_{\theta r}(x_0-x)}  \left\{ u(x+y) -\left(1-\frac{\theta}{2}\right)^{-\gamma}u(x_0)\right\}^+\frac{l(|y|)}{|y|^n} dy
\end{align*}
in the viscosity sense, 
 where $v$ satisfies   $\cM^-v=\cM^-[-u] \leq L(\rho_0R)$ on $B_{2R}$ in the viscosity sense.  

Consider  the largest number $\beta>0$ such that 
  $$u(x)\geq g_\bt(x):= \bt\left(1-\frac{|4x|^2}{R^2}\right)^{+},$$ and let $x_1\in   B_{\frac{R}{4}} $ be a point  such that $u(x_1)=g_\bt (x_1).$ We observe that $\bt\leq 1$ since $u(0)\leq1.$ 
Using Lemma \ref{lem-kernel-l}, we have 
\begin{align*}
(2-\sm)\int_{\R^n} \mu^-(u,x_1,y)\frac{l(|y|)}{|y|^n}dy&\leq(2-\sm)\int_{\R^n} \mu^-(g_\bt,x_1,y)\frac{l(|y|)}{|y|^n}dy \\
&=(2-\sm)\left\{\int_{B_{\rho_0 R}} \mu^-(g_\bt,x_1,y)\frac{l(|y|)}{|y|^n}dy+\int_{\R^n\setminus B_{\rho_0R}} \mu^-(g_\bt,x_1,y)\frac{l(|y|)}{|y|^n}dy\right\}\\
&
\leq C(2-\sm)\bt \int_{B_{\rho_0 R}}  \frac{  |y|^2}{R^2}\frac{l(|y|)}{|y|^n}dy+2(2-\sm)\bt\int_{\R^n\setminus B_{\rho_0R}}\frac{l(|y|)}{|y|^n}dy\\
&
\leq \frac{C\bt}{R^2}2a_0 {\rho_0^2R^2l(\rho_0 R )}+2\beta\frac{2-\sm}{\sm}\left\{ L(\rho_0 R)+2a_{\infty}\right\}\\
&
\leq  {C\bt a_0 \rho_0^2}  {L(\rho_0 R)}+\frac{8\bt}{\sm_0} {L(\rho_0 R)}\leq C\bt L(\rho_0R)\leq CL(\rho_0R),
\end{align*}  
 where we recall that $0<\rho_0\leq \rho_1$; see    \eqref{eq-choice-rho1}. 
Since $\cM^-u\leq L(\rho_0R)$ on $B_{2R}$ in the viscosity sense,   it follows 
that  
\begin{equation*}
(2-\sm)\int_{\R^n} \mu^+(u,x_1,y)\frac{l(|y|)}{|y|^n}dy\leq CL(\rho_0R),
\end{equation*}
which asserts that
\begin{equation}\label{eq-harnack-proof-x1}
(2-\sm)\int_{\R^n} \left\{u(x_1+y)-2\right\}^+\frac{l(|y|)}{|y|^n}dy\leq CL(\rho_0R),
\end{equation}
   where  we note that $u(x_1)\leq \beta\leq 1$ and $u(x_1-y)>0$ for   any $y\in\R^n.$
 
 We may assume that $u(x_0)\geq2,$ otherwise $\ap$ is uniformly bounded.  In order to  estimate  $\cM^-w$ in $B_{\frac{\theta }{2}r}(x_0),$  
   we consider that    for $x\in B_{\frac{\theta }{2}r}(x_0)$
 \begin{align*}
 &(2-\sm) \int_{\R^n\setminus B_{\theta r}(x_0-x)}  \left\{ u(x+y) -\left(1-\frac{\theta}{2}\right)^{-\gamma}u(x_0)\right\}^+\frac{l(|y|)}{|y|^n} dy\\
 =&(2-\sm) \int_{\R^n\setminus B_{\theta r}(x_0-x)}  \left\{ u(x_1+x+y-x_1) -\left(1-\frac{\theta}{2}\right)^{-\gamma}u(x_0)\right\}^+\frac{l(|x+y-x_1|)}{|x+y-x_1|^n}\\
 &\qquad\qquad\cdot\left( \frac{|x+y-x_1|^n}{|y|^n } \frac{l(|y|)}{l(|x+y-x_1|)}  \right)dy. 
\end{align*}
 Since we see that   for    $x\in B_{\frac{\theta }{2}r}(x_0)$ and  $y\in \R^n\setminus B_{\theta r}(x_0-x) ,$  
  $$\frac{|x+y-x_1|^n}{|y|^n } \frac{l(|y|)}{l(|x+y-x_1|)}  dy \leq  a_0 a_\infty\left(\frac{6R}{\theta r}\right)^{n+\sm+\max(\delta,\delta')},$$
    it  follows from \eqref{eq-harnack-proof-x1} that  for  $x\in B_{\frac{\theta }{2}r}(x_0)$
  \begin{align*}
  \cM^-w(x)&\leq  L(\rho_0R)+ 2\Ld a_0a_\infty 3^{n+3}\left(\frac{2R}{\theta r}\right)^{n+\sm+\max(\delta,\delta')} CL(\rho_0 R)\\
  &\leq C\left(\frac{2R}{\theta r}\right)^{n+\sm+\delta}  L(\rho_0 R)\leq C\left(\frac{2R}{\theta r}\right)^{n+2}  L\left(\rho_0 \frac{\theta r}{2}\right)
  \end{align*}
  owing to monotonicity of the function  $L.$ 
  
Now we apply  the weak Harnack inequality  to $w$ in $B_{\frac{\theta }{2}r}(x_0)$ to obtain that 
\begin{align*}
&\left|\left\{ u<\frac{u(x_0)}{2}\right\}\cap B_{\theta r/4}(x_0)\right| =\left|\left\{ w> \left(\left(1-\frac{\theta}{2}\right)^{-\gamma}-\frac{1}{2}\right)u(x_0)\right\}\cap B_{\theta r/4}(x_0)\right| \\
&\leq C\left(\theta r\right)^n\left(w(x_0)+C\left(\frac{2R}{\theta r}\right)^{n+\sm+\delta} \right)^\ep  \left(\left(1-\frac{\theta}{2}\right)^{-\gamma}-\frac{1}{2}\right)^{-\ep}{u(x_0)^{-\ep}}\\
&= C\left(\theta r\right)^n\left(\left(\left(1-\frac{\theta}{2}\right)^{-\gamma}-1\right)u(x_0)+C\left(\frac{2R}{\theta r}\right)^{n+2} \right)^\ep  \left(\left(1-\frac{\theta}{2}\right)^{-\gamma}-\frac{1}{2}\right)^{-\ep}{u(x_0)^{-\ep}} \\
&\leq C\left(\theta r\right)^n\left(\left(\left(1-\frac{\theta}{2}\right)^{-\gamma}-1\right)^\ep+C\left(\frac{2R}{\theta r}\right)^{(n+2)\ep} \frac{1}{u(x_0)^\ep }\right)\\
&\leq C\left(\theta r\right)^n\left(\left(\left(1-\frac{\theta}{2}\right)^{-\gamma}-1\right)^\ep+C\left(\frac{2R}{\theta r}\right)^{(n+2-\gamma)\ep}\left(\frac{\theta}{4}\right)^{-\gamma\ep}  \ap^{-\ep}\right)\\
&\leq C\left(\theta r\right)^n\left(\left(\left(1-\frac{\theta}{2}\right)^{-\gamma}-1\right)^\ep+ \theta^{-\gamma\ep}  \ap^{-\ep}\right)
\end{align*}   
since 
$ u(x_0)=\ap(R/2r)^{\gamma}$ and $\gamma= (n+2)/\ep.$   We choose a uniform constant   $\theta>0$  sufficiently small so that 
$$C\left(\theta r\right)^n \left(\left(1-\frac{\theta}{2}\right)^{-\gamma}-1\right)^\ep \leq \frac{1}{4}|B_{\theta r/4}(x_0)| .$$  If $\ap>0$ is sufficiently  large, then we have  
$$C\left(\theta r\right)^n \theta^{-\gamma\ep}  \ap^{-\ep} \leq  \frac{1}{4}|B_{\theta r/4}(x_0)| ,$$
which  implies that
$$\left|\left\{ u<\frac{u(x_0)}{2}\right\}\cap B_{\theta r/4}(x_0)\right|  \leq \frac{1}{2}|B_{\theta r/4}(x_0)|. $$
On the other hand,  according to   \eqref{eq-harnack-pf-2}, we have that   for large $\ap>0$   
 $$ \left|\left\{u> \frac{u(x_0)}{2}\right\}\cap B_{\theta r/4}(x_0)\right|    \leq C  \ap^{-\ep} |B_{\theta r/4}(x_0)|<\frac{1}{2}|B_{\theta r/4}(x_0)|,$$ 
 which is a  contradiction.  Therefore, we conclude that $\ap>0$ is uniformly bounded and that   $\displaystyle\sup_{B_{\frac{R}{2}}}u\leq \ap 2^{\gamma},$   which completes  the proof.
\end{proof}


\subsection{H\"older continuity}
  From the Harnack inequality, we obtain the following H\"older regularity of the viscosity solutions to fully nonlinear elliptic integro-differential equations with respect to $\fL(\ld,\Ld,l)$.  
\begin{thm}[H\"older continuity]\label{thm-Holder}
For $0<R<1,$ and $C_0>0,$ let $u\in C(B_{2R})$ be a nonnegative   function in $\R^n$ such that  
 $$\cM^-_{\fL(\ld,\Ld,l)} u\leq C_0L(\rho_0 R), \quad\mbox{and}\quad\cM^+_{\fL(\ld,\Ld,l)}u\geq -C_0L(\rho_0 R)\quad\mbox{in $B_{2R}$}$$ in the viscosity sense.  
Then we have 
$$R^{\ap}\,[u]_{\ap, B_{R}}\leq C\left( \|u\|_{L^{\infty}(\R^n)}+C_0\right)$$
 where   $[u]_{\ap, B_{R}}$ stands for  the $\ap$-H\"older seminorm on $B_R$, and  uniform constants $\rho_0,\ap\in(0,1)$ and   $C>0$ depend only on  $n,\ld,\Ld,   $  $ a_0,a_\infty,$ $ \rho$ and $\sm_0$. 
\end{thm}

\subsection{$C^{1,\alpha}$ estimate} \label{subsec-C1alpha}
In this subsection, we present an interior $C^{1,\ap}$ estimate for viscosity solutions to the elliptic  integro-differetial operators 
as a important consequence of the H\"older estimate. 
To   apply the incremental quotients technique iteratively in the nonlocal setting,  the  cancellation condition \eqref{eq-kernel-c1alpha} below  for the kernels 
at infinity  is assumed;  refer to \cite[Section 13]{CS1}.  
For contants $\theta_0>0$ and $D_0>0,$  we define   
$\fL_1\left(\ld,\Ld, l;\theta_0,D_0\right)$ 
   by  the class  of the following linear integro-differential  operators 
  with the  kernels $K$:
\begin{equation*}
\cL u(x)=\int_{\R^n}\mu(u,x,y)  K(y) dy,
\end{equation*}
such that  
\begin{equation*}
  (2-\sm)\ld\f{ l (|y|)}{|y|^{n}}\le
K(y)
\le (2-\sm)\Ld\f{ l(|y|)}{|y|^{n}},
\end{equation*}
and 
\begin{equation}\label{eq-kernel-c1alpha} 
  \int_{\R^n\setminus B_{\theta_0}}\frac{|K(y)-K(y-h)|}{|y|}dy \leq D_0, \quad \forall |h|<\frac{\theta_0}{2}. 
  \end{equation}

\begin{thm}There is a uniform constant $\theta_0>0$ (depending  only on  $n,\ld,\Ld,   $  $ a_0,a_\infty,$ $ \rho,$  $\sm_0$)   such that if 
    $u\in C(B_{1})$ is  a bounded, nonnegative   function in $\R^n$ such that  
 $\cI u=0$ in $B_1$ in the viscosity sense for      an elliptic operator  $\cI$ with respect to $\fL_1(\ld,\Ld,l; \theta_0,D_0), $  
then  we have 
$$ \|u\|_{C^{1,\ap}(B_{1/2})}\leq C\left( \|u\|_{L^{\infty}(\R^n)}+|\cI 0|\right),$$
 where    $\ap\in(0,1)$ and   $C>0$ depend only on  $n,\ld,\Ld,   $  $ a_0,a_\infty,$ $ \rho,$  $\sm_0,$ and $D_0.$ 
\end{thm}

Making use of the incremental quotients, we   obtain   $C^{1,\ap}$-estimate for nonlocal operators. The  uniform H\"older estimate in Theorem \ref{thm-Holder}  is applicable to 
$$w_h(x):=\frac{u(x+h)-u(x)}{|h|^\ap}$$
for   any small   vector $h\in\R^n$ when \eqref{eq-kernel-c1alpha} holds. Indeed, we introduce the cut-off function $\eta$  supported in a smaller ball, and divide the incremental quotient $w_h$ into two functions $w_{h,1}:=\eta w_h$ and $w_{h,2}:=(1-\eta)w_h.$ With the help of \eqref{eq-kernel-c1alpha},       we deal with the   incremental quotient of the kernel $K$ replacing the       incremental quotient of $u$  
in order  to  show that $|\cL w_{h,2}|$ is bounded by $C\|u\|_{L^{\infty}(\R^n)}$. 
So we apply Theorem \ref{thm-Holder} to $w_{h,1}$ to deduce $C^{2\ap}$ for the H\"older exponent $\ap>0$ in  Theorem \ref{thm-Holder}.
 Employing  the    procedure $[1/\ap]$ times, it follows that $u$ is Lipschitz continuous.  By applying the  previous argument to  the Lipschitz quotient of $u$, we deduce the uniform   $C^{1,\ap}$-estimate.   
Note that the $C^{1,\ap}$-estimate is not scale-invariant since it relies on the values $\theta_0$ and $D_0.$ 

\subsection{Truncated kernels at infinity}\label{subsec-truncated}

In this subsection, we are concerned with the elliptic integro-differential operators associated with the symmetric kernels  satisfying Property \ref{hypo-kernel-l} near zero 
which  may not satisfy Property \ref{hypo-kernel-l-infty} at infinity.  This subsection corresponds to \cite[Section 14]{CS1}  which involves in the results of the fractional Laplacian type integro-differential operators. 
  Consider the   linear integro-differential operator $\cL$
  \begin{equation*}
\cL u(x)=\int_{\R^n}\mu(u,x,y)  K(y) dy,
\end{equation*}
   with  the    nonnegative kernel $K$ which is    split by 
$$ K(y)= K_1(y)+K_2(y)\geq 0\quad\mbox{in $\R^n$},$$
where  the linear integro-differential operator $\cL_1$ with the kernel $K_1$ belongs to $\fL(\ld,\Ld,l),$ and $\|K_2\|_{L^1(\R^n)}\leq \kappa$ for $\kappa\geq0.$ For $\kappa\geq0,$  we denote   by $\tilde \fL(\ld,\Ld,l,\kappa)$ the class of   all the   linear integro-differential operators above. 
Using  Lemma \ref{lem-kernel-l},     we see that the    truncated  kernel   $K$ at infinity  satisfying  $$(2-\sm)\ld\f{ l(|y|)}{|y|^{n}} \chi_{B_1(0)}\leq K(y)\leq(2-\sm)\Ld\f{ l(|y|)}{|y|^{n}} \chi_{B_1(0)}$$
is one of  the typical   kernels  for   the linear integro-differential  operators belonging to the  class $\tilde\fL\left(\ld,\Ld,l, 2(2-\sm)a_\infty/\sm\right)$. 
 It is obvious  that the larger class $\tilde\fL(\ld,\Ld,l,\kappa)$ coincides with $\fL(\ld,\Ld,l)$ for  $\kappa=0.$  The Pucci type  extremal operators with respect to the class $\tilde \fL(\ld,\Ld,l,\kappa)$ are defined as 
\begin{equation*}
\begin{split}
\cM^+_{  \tilde\fL\left(\ld,\Ld, l,\kappa\right)} u&:=\sup_{\cL\in \tilde\fL\left(\ld,\Ld, l,\kappa\right)}\cL u,\\
\cM^-_{ \tilde \fL\left(\ld,\Ld, l,\kappa\right)} u&:=\inf_{\cL\in\tilde\fL\left(\ld,\Ld, l,\kappa\right)}\cL u.
\end{split}
\end{equation*}
     The same  argument  as in \cite[Lemma 14.1]{CS1} provides the following lemma regarding the relation between the Pucci type   operators with respect to the classes $\tilde\fL(\ld,\Ld,l,\kappa)$ and  $\fL(\ld,\Ld,l).$
\begin{lemma} \label{lem-tr-kernel}
Let $u$ be a bounded function in $\R^n$ and $C^{1,1}$ at $x.$ Then we have 
$$\cM^-_{\tilde\fL(\ld,\Ld,l,\kappa)}u(x)\geq \cM^-_{\fL(\ld,\Ld,l)}u(x)-4\kappa \|u\|_{L^{\infty}(\R^n)},
$$
and 
$$\cM^+_{\tilde\fL(\ld,\Ld,l,\kappa)}u(x)\leq \cM^+_{\fL(\ld,\Ld,l)}u(x)+4\kappa \|u\|_{L^{\infty}(\R^n)}.
$$
\end{lemma}
Applying  Theorem \ref{thm-Holder}   combined  with    Lemma \ref{lem-tr-kernel}, we deduce the H\"older estimate for the  elliptic  integro-differential operators associated with truncated kernels at infinity. 
\begin{thm} \label{thm-Holder-truncated}
For $0<R<1,$ and $C_0>0,$ let $u\in C(B_{2R})$ be a bounded, nonnegative   function in $\R^n$ such that  
 $$\cM^-_{\tilde\fL(\ld,\Ld,l,\kappa)} u\leq C_0L(\rho_0 R), \quad\mbox{and}\quad\cM^+_{\tilde\fL(\ld,\Ld,l,\kappa)}u\geq -C_0L(\rho_0 R)\quad\mbox{in $B_{2R}$}$$ in the viscosity sense.  
Then we have 
$$R^\ap\,[u]_{{\ap},B_{R}}\leq C\left\{(1+4\kappa)\|u\|_{L^{\infty}(\R^n)}+C_0\right\}, $$
 where    uniform constants $\rho_0,\ap\in(0,1)$ and   $C>0$ depend only on  $n,\ld,\Ld,   $  $ a_0,a_\infty,$ $ \rho$ and $\sm_0$. 
\end{thm}


\section{Uniform regularity   estimates  for  certain   integro-differential operators   as $\sm\to 2-$    }\label{sec-uniform-harnack}
In this section,     we   study the uniform   Harnack inequality and H\"older estimate for the elliptic  integro-differential operators associated with   the certain  regularly varying kernels  at zero, where the regularity estimates   remain  uniform      
as the order $\sm\in(0,2)$ of the operator tends to $2.$ Consider  
    a  measurable function  $l_0:(0,+\infty)\to(0,+\infty)$ 
   which  stays locally bounded away from $0$ and $+\infty$,    and  is   slowly varying at zero and infinity.  
 We may assume that  $l_0(1)=1.$
 For $\sm\in(0,2), $ we define 
 \begin{equation}\label{eq-l-sigma}
 l_\sm(r):=r^{-\sm}l_0(r)^{2-\sm},\quad\forall r>0. 
 \end{equation}
It is easy to check that the function $l_\sm$        varies  regularly  at zero and infinity  with index $-\sm\in(-2,0).$
As seen in Subsection \ref{subsec-operators},  let $\fL\left(\ld,\Ld, l_\sm\right)$ denote   the class of all linear   integro-differential operators  
\begin{equation*}
\cL u(x)=\int_{\R^n}\mu(u,x,y)K(y)\,dy
\end{equation*} 
 with  the  kernels  $K$     satisfying
 \begin{equation*}
(2-\sm)  \ld\f{ l_\sm (|y|)}{|y|^{n}}\le
K(y)\le(2-\sm) \Ld\f{ l_\sm(|y|)}{|y|^{n}},
\end{equation*}
 where 
 $\mu(u,x,y):=u(x+y)+u(x-y)-2u(x).$  For a given constant $\sm_0\in(0,2)$ and a function  $l_0,$ 
  the  uniform Harnack inequality and H\"older estimate  for   fully   nonlinear    elliptic integro-differential operators  with respect to the class $\fL(\ld,\Ld,l_\sm)$ for   $\sm\in[\sm_0,2) $ are established.  In fact, 
 once  it is proved  that the function $l_\sm$ satisfies  Properties \ref{hypo-kernel-l}  and \ref{hypo-kernel-l-infty}  with  uniform constants $a_0, a_\infty\geq 1$ and $\rho\in(0,1)$ with respect to $\sm\in[\sm_0,2),$ 
the uniform regularity estimates follow from the results of  Section \ref{sec-harnack}.  Here, the regularity  estimates depend only on $n,\ld,\Ld, \sm_0,$ and the given function $l_0.$  Therefore,  it suffices to prove the following proposition in order to obtain   Theorem \ref{thm-Harnack-l-sigma-intro}. 

 \begin{prop} \label{prop-l-sigma}
 For a given $\sm_0\in(0,2),$ let  
 $$\delta_0:=\min\left(  \frac{\sm_0}{2(2-\sm_0)}, \frac{1}{2}\right).$$
  For $\sm\in[\sm_0,2),$ let $l_\sm$ be defined as   \eqref{eq-l-sigma}.  
  Then  $l_\sm$ satisfies Properties \ref{hypo-kernel-l}   and \ref{hypo-kernel-l-infty} with  uniform constants $a_0, a_\infty\geq 1$ and $\rho\in(0,1)$  with respect to $\sm\in[\sm_0,2),$ 
  where the constants  $a_0,a_\infty$ and $\rho$ depend only on  $\sm_0$, and  the slowly varying function $ l_0$.  

 \end{prop}  
\begin{proof} Since $l_0$ varies  slowly at zero and infinity,     there exist    $c_0\geq 1  $ and  $c_\infty\geq 1  $  such that  
$$  \frac{l_0(s)}{l_0(r)}\leq  c_0 \max\left\{  \left(\frac{s}{r}\right)^{\delta_0}, \left(\frac{s}{r}\right)^{-\delta_0}\right\}\quad\forall r,s\in(0,1],$$
$$  \frac{l_0(s)}{l_0(r)}\leq  c_\infty \max\left\{  \left(\frac{s}{r}\right)^{\delta_0}, \left(\frac{s}{r}\right)^{-\delta_0}\right\}\quad\forall r,s\in[1,+\infty)$$
from Potter's theorem; see Appendix \ref{sec-regular-variation}.    
This implies that  
$$  \frac{l(s)}{l(r)}\leq  c_0^{2-\sm} \max\left\{  \left(\frac{s}{r}\right)^{-\sm+\delta_0(2-\sm)}, \left(\frac{s}{r}\right)^{-\sm-\delta_0(2-\sm)}\right\}\quad\forall r,s\in(0,1],$$
and 
$$  \frac{l(s)}{l(r)}\leq  c_\infty^{2-\sm} \max\left\{  \left(\frac{s}{r}\right)^{-\sm+\delta_0(2-\sm)}, \left(\frac{s}{r}\right)^{-\sm-\delta_0(2-\sm)}\right\}\quad\forall r,s\in[1,+\infty).$$ 
By choosing $a_0:=c_0^2, a_\infty:=c_\infty^2,$ and $\delta=\delta'=\delta_0(2-\sm),$  the property (a) in Property \ref{hypo-kernel-l} and Property \ref{hypo-kernel-l-infty} hold  since $$\delta=\delta'\leq  \frac{\sm_0}{2(2-\sm_0)}(2-\sm)\leq \frac{\sm_0}{2}\leq \frac{\sm}{2}\quad\forall \sm\in[\sm_0,2).$$  
  Lastly,    the property (b) in Property \ref{hypo-kernel-l}  will be proved   in the following Lemma \ref{lem-l-sigma}. 
\end{proof} 

\begin{lemma}   \label{lem-l-sigma}
 Under the same assumption as in Proposition \ref{prop-l-sigma}, we have that 
\begin{equation*}
  \frac{{\sm} {\int_{r}^1s^{-1}  {l_\sm(s)}ds}}{l_\sm(r)} \to 1 \quad\mbox{as $r\to0+$
}
\end{equation*}
  uniformly with respect to $\sm\in[\sm_0,2). $ In particular, there exists a uniform constant $\rho\in(0,1)$ such that for any  $\sm\in[\sm_0,2),$
  $$ \frac{1}{2}\leq \frac{{\sm} {\int_{r}^1 s^{-1} {l_\sm(s)}ds}}{l_\sm(r)} =\frac{L_\sm(r)}{l_\sm(r)} \leq  {2},\qquad\forall r\in(0,\rho) ,$$
  where $\displaystyle L_\sm(r):=  {\sm}\int_{r}^1  {s^{-1}l_\sm(s)}ds.$
\end{lemma}
\begin{proof}
For $0<r<1,$ we   rewrite 
\begin{align*}
\frac{\int_r^1 s^{-1-\sm} l_0(s)^{2-\sm}ds }{ r^{-\sm}l_0(r)^{2-\sm}}=\int_1^{ {1}/{r}} t^{-1-\sm}\left(\frac{l_0(tr)}{l_0(r)}\right)^{2-\sm} dt.
\end{align*} 
Note that  for $t\in(1,1/r)$
$$\left(\frac{l_0(tr)}{l_0(r)}\right)^{2-\sm}\leq  c_0^{2-\sm}t^{\delta_0(2-\sm)},$$
  in the proof  of Proposition \ref{prop-l-sigma}. 
So the integrand is bounded by 
$$c_0^{2} t^{-1-\sm+\delta_0(2-\sm)}$$
which is integrable since $\sm-\delta_0(2-\sm)\geq \sm/2\geq  \sm_0/2.$ Thus it follows from  the Dominated Convergence Theorem  that $\int_1^{ {1}/{r}} t^{-1-\sm}\left(\frac{l_0(tr)}{l_0(r)}\right)^{2-\sm} dt$ converges to $\int_1^{+\infty} t^{-1-\sm} dt =\frac{1}{\sm}$  as $r\to0+$  since $l_0$  varies slowly at zero. 
Now, it remains  to show the uniform convergence with respect to $\sm\in[\sm_0,2).$  Let $\vep\in(0,1)$ be given. 
For a small uniform constant  $r_0\in(0,1)$ to  be chosen later, we have  
\begin{align*}
\sm\int_1^{ {1}/{r}} t^{-1-\sm}\left(\frac{l_0(tr)}{l_0(r)}\right)^{2-\sm} dt-1
&= \sm\int_1^{ {1}/{r_0}} t^{-1-\sm}\left\{\left(\frac{l_0(tr)}{l_0(r)}\right)^{2-\sm} -1\right\}dt\\
&+\sm\int_{1/r_0}^{ {1}/{r}} t^{-1-\sm} \left(\frac{l_0(tr)}{l_0(r)}\right)^{2-\sm}  dt-\sm\int_{ {1}/{r_0}}^{\infty} t^{-1-\sm} dt\\
&=: I_1+I_2+I_3.
\end{align*} 
We select $r_0>0$ sufficiently  small so that  for $0<r<r_0,$
\begin{align*}
I_2&\leq \sm c_0^2\int_{1/r_0}^{ {1}/{r}} t^{-1-\sm +\delta_0(2-\sm)}  dt= \sm c_0^2\frac{1}{\sm-\delta_0(2-\sm)}\left\{ r_0^{\sm-\delta_0(2-\sm)}- r^{\sm-\delta_0(2-\sm)}\right\}\\
&\leq \sm c_0^2\frac{2}{\sm }  r_0^{\sm/2}\leq  2 c_0^2   r_0^{\sm_0/2}< \frac{\vep}{2},
\end{align*} 
and hence   
\begin{align*}
|I_3|&= \sm  \int_{1/r_0}^{ \infty} t^{-1-\sm  }  dt\leq    r_0^{\sm}  \leq r_0^{\sm_0/2}< \frac{\vep}{4}.
\end{align*}  
Now we claim  that  for a fixed $r_0>0,$
$$ I_1:= \sm\int_1^{ {1}/{r_0}} t^{-1-\sm}\left\{\left(\frac{l_0(tr)}{l_0(r)}\right)^{2-\sm} -1\right\}dt \to 0\quad\mbox{as $r\to0+$}$$
uniformly with respect to $\sm\in[\sm_0,2). $
According to  the Uniform Convergence Theorem in \cite[Theorem 1.5.2]{BGT}, we    have that
\begin{equation*}
\frac{l_0(tr)}{l_0(r)}  \to 1 \quad\mbox{as $r\to 0+\,\,$ uniformly for   $t\in\left[1, 1/r_0\right].$ }
\end{equation*}
 Then  it follows that   $\displaystyle \left(\frac{l_0(tr)}{l_0(r)}\right)^{2-\sm}$ uniformly converges  to $1$ as $r\to0+$ for  $t\in\left[1, 1/r_0\right]$ and $\sm\in(0,2).$ 
 Namely, there exists  a uniform constant  $\rho\in(0,1)$ with respect to $\sm\in[\sm_0,2),$ depending only on $l_0$ and $\sm_0,$ such that 
   $$\left|\left(\frac{l_0(tr)}{l_0(r)}\right)^{2-\sm} -1\right| <  \frac{\vep}{4},\quad\forall  r\in(0,\rho),\,\, t\in[1,1/r_0].$$ 
   Thus, we have that for any $0<r<\rho,$ 
\begin{align*}
 \sm\int_1^{ {1}/{r_0}} t^{-1-\sm}\left\{\left(\frac{l_0(tr)}{l_0(r)}\right)^{2-\sm} -1\right\}dt&< \frac{\sm \vep}{4} \int_1^{1/r_0} t^{-1-\sm}dt\\ 
 &\leq \frac{  \vep}{4} \left(1- r_0^{\sm}\right) 
 < \frac{\vep}{4}.
\end{align*}     
Therefore,  for any $0<r<\rho,$ we conclude  that  
\begin{align*}
\left|\sm\int_1^{ {1}/{r}} t^{-1-\sm}\left(\frac{l_0(tr)}{l_0(r)}\right)^{2-\sm} dt-1\right| <\vep,
\end{align*} 
which  finishes the proof. 
\end{proof}

 \appendix

\section{Regular   variations}\label{sec-regular-variation}

We recall    regularly varying functions at zero  and their properties. The results of regularly varying  functions at infinity  are  
established in  \cite{BGT}, which  are  simple inversions of those for regularly varying functions at zero.    
\begin{definition}[Regular and slow variations] Let  $l: (0, 1) \to (0, +\infty) $ be  a  measurable function. 
\begin{enumerate}[(i)]
\item  A       function $l: (0, 1) \to (0, +\infty) $ is said to vary regularly at zero with index $\alpha\in\R$  if for every $\kappa> 0$
$$\lim_{r \rightarrow 0+}\frac{l(\kappa r)}{l (r)}=\kappa^{\ap}.$$
\item  A regularly varying function is called to be slowly varying if its index $\ap$ is zero.
\end{enumerate}

\end{definition}
 We state the important   properties of  regularly and slowly varying functions     used in this paper as a lemma. The proofs  and more details for regular and slow variations  can be found in    \cite{BGT}; see also \cite[Appendix A]{KM}. 
      
\begin{lemma}
\begin{enumerate}[(i)]
Let  $l: (0, 1) \to (0, +\infty) $ be  a  measurable function.  
\item Any function $l $ that varies regularly with index $\ap\in\R$  is of the form
  $$l(r)=r^{\ap}l_0(r)$$  for some  slowly varying  function $l_0.$  
\item Let $l$ be a   regularly  varying function with index $-\ap\leq 0$ which is locally bounded away from $0$ and $+\infty$.  Then  Potter's theorem  \cite[Theorem 1.5.6]{BGT}   asserts that for any 
  $\delta>0,$ there exists   $A_ \delta\geq 1  $ such that for $0<r,s<1$   
$$  \frac{l(s)}{l(r)}\leq  A_\delta \max\left\{  \left(\frac{s}{r}\right)^{-\ap+\delta}, \left(\frac{s}{r}\right)^{-\ap-\delta}\right\}.$$
\item  Let $l$ be  slowly varying and $\beta>-1.$ Then Karamata's theorem    \cite[Proposition 1.5.8]{BGT} asserts that 
$$\lim_{r\to0+}\frac{\int_0^r s^\beta l(s)ds }{ r^{\beta+1 } l(r)}=\frac{1}{\beta+1}.$$
\item Let  $l$  be a     regularly varying function  with index $-\ap<0.$ Then   \cite[Theorem 1.5.11]{BGT}  states that 
$$\lim_{r\to0+}\frac{\int_r^1s^{-1} l(s)ds }{l(r)}=\frac{1}{\ap}.$$
This implies that if  $l$ varies regularly with index $\ap<0,$ so does  the function $r\mapsto \int_r^1s^{-1} l(s)ds.$
\item Let  $l$ be a regularly varying function with index $\ap\in\R.$ Then the  
Uniform Convergence Theorem   \cite[Theorem 1.5.2]{BGT}  asserts that 
$$\frac{l(\kappa r)}{l(r)}\to\kappa^\ap\quad\mbox{as $r\to0+\,\,\,\,$ uniformly in $\kappa\in[a,b]$ }$$ 
for each $[a,b]\subset(0,+\infty).$
\end{enumerate}
\end{lemma}

   {\bf Acknowledgement} 
  The authors would like to thank  Prof. Moritz Kassmann  for suggesting   the problem considered in this  paper.  
 Ki-Ahm Lee  was supported by the National Research Foundation of Korea(NRF) grant funded by the Korea government(MSIP) (No.2014R1A2A2A01004618). 
Ki-Ahm Lee also hold a joint appointment with the Research Institute of Mathematics of Seoul National University. 
   
  



\end{document}